\documentclass[11pt,reqno]{amsart}
\usepackage{amsmath,amssymb,amsthm}
\usepackage{booktabs}
\usepackage[shortlabels]{enumitem}
\usepackage{graphicx}
\usepackage[latin1]{inputenc}
\usepackage{mathtools}
\usepackage{thmtools}
\usepackage{url}
\usepackage[usenames,svgnames]{xcolor}

\usepackage[pagebackref]{hyperref} 
\usepackage[all]{hypcap} 
\usepackage[alphabetic,backrefs,nobysame]{amsrefs} 

\title[A second proof of the Shareshian--Wachs conjecture]{A second proof of the Shareshian--Wachs conjecture, by way of a new Hopf algebra}
\date{\today}

\author{Mathieu Guay-Paquet}
\email{mathieu.guaypaquet@lacim.ca}
\address{%
LaCIM\\
Universit\'e du Qu\'ebec \`a Montr\'eal\\
201 Pr\'esident-Kennedy\\
Montr\'eal QC\ \ H2X~3Y7\\
Canada}

\hypersetup{
  breaklinks,colorlinks,
  citecolor=Green,
  linkcolor=BlueViolet,
  urlcolor=DarkBlue,
  pdftitle={A second proof of the Shareshian--Wachs conjecture, by way of a new Hopf algebra},
  pdfauthor={Mathieu Guay-Paquet},
}

\declaretheorem[numberlike=equation,style=plain]{theorem}
\declaretheorem[numberlike=equation,style=plain]{lemma}
\declaretheorem[numberlike=equation,style=plain]{corollary}
\declaretheorem[numberlike=equation,style=plain]{proposition}
\declaretheorem[numberlike=equation,style=plain]{conjecture}

\declaretheorem[numberlike=equation,style=definition]{remark}

\declaretheorem[numberlike=equation,style=definition]{question}
\declaretheorem[numberlike=equation,style=definition]{algorithm}

\newcommand{\arxiv}[1]{\href{http://arxiv.org/abs/#1}{arXiv:#1}}
\newcommand{\doi}[1]{\href{http://dx.doi.org/#1}{doi:#1}}

\graphicspath{{figs/}}

\DeclarePairedDelimiter{\abs}{|}{|}
\DeclarePairedDelimiter{\set}{\{}{\}}
\DeclarePairedDelimiter{\paren}{(}{)}
\DeclarePairedDelimiter{\gen}{\langle}{\rangle}

\DeclareMathOperator{\csf}{CSF}

\DeclareMathOperator{\qsym}{\mathcal{Q}Sym}
\DeclareMathOperator{\sym}{Sym}
\DeclareMathOperator{\og}{\mathcal{G}}
\DeclareMathOperator{\dyck}{\mathcal{D}}
\DeclareMathOperator{\flag}{Flag}
\DeclareMathOperator{\hess}{Hess}
\DeclareMathOperator{\tr}{Trace}
\DeclareMathOperator{\frob}{Frob}
\DeclareMathOperator{\stat}{stat}

\newcommand{\CC}{\mathbb{C}}
\newcommand{\PP}{\mathbb{P}}
\newcommand{\xx}{\mathbf{x}}
\newcommand{\HH}{\mathcal{H}}
\newcommand{\LL}{\mathbf{L}}
\newcommand{\RR}{\mathbf{R}}
\newcommand{\TT}{\mathcal{T}}
\newcommand{\OO}{\mathcal{O}}

\newcommand{\id}{\mathrm{id}}
\newcommand{\olex}{\mathbin{\vec{\oplus}}}
\newcommand{\bigolex}{\mathop{\overrightarrow{\bigoplus}}}
\newcommand{\zq}{\zeta_{\mathcal{Q}}}
\newcommand{\sll}[1][]{S(\LL_{#1})}
\newcommand{\srr}[1][]{S(\RR_{#1})}
\newcommand{\slr}{S(\LL \leftarrow \RR)}
\newcommand{\one}{\mathbf{1}}

\renewcommand{\emph}{\textbf}

\begin{document}

\begin{abstract}
This is a set of working notes which give a second proof of the Shareshian--Wachs conjecture, the first (and recent) proof being by Brosnan and Chow in November 2015.
The conjecture relates some symmetric functions constructed combinatorially out of unit interval graphs (their $q$-chromatic quasisymmetric functions), and some symmetric functions constructed algebro-geometrically out of Tymoczko's representation of the symmetric group on the equivariant cohomology ring of a family of subvarieties of the complex flag variety, called regular semisimple Hessenberg varieties.
Brosnan and Chow's proof is based in part on the idea of deforming the Hessenberg varieties.
The proof given here, in contrast, is based on the idea of recursively decomposing Hessenberg varieties, using a new Hopf algebra as the organizing principle for this recursion.
We hope that taken together, each approach will shed some light on the other, since there are still many outstanding questions regarding the objects under study.
\end{abstract}

\maketitle

\setcounter{tocdepth}{1}
\tableofcontents

\section{Introduction}\label{sec:intro}

The goal of this document is to give all of the necessary ingredients to prove the Shareshian--Wachs conjecture (see \autoref{sec:hess} for the notation):
\begin{conjecture}[\cite{shareshian-wachs12}*{Conjecture 5.3}]
  For all Hessenberg functions $h$,
  \begin{equation}
    \omega\paren[\big]{\frob_q\paren[\big]{\sll, H^*_T(h), \CC[\LL]}} = \csf_q\paren[\big]{G(h)}.
  \end{equation}
\end{conjecture}
As noted above, we are not the first ones to do this; Brosnan and Chow \cite{brosnan-chow15} have recently given a proof.
As part of their proof, they uncover and use much structure on the geometric side of things, not just for the case of regular semisimple Hessenberg varieties, but also relating it to the case of regular (not necessarily semisimple) Hessenberg varieties.
For the proof here, we instead focus on recursions in the combinatorial defining data for regular semisimple Hessenberg varieties, which seems like a complementary approach.

Quasisymmetric functions are ubiquitous in algebraic combinatorics for the following reason: nice combinatorial objects can usually be combined and broken apart in a way that leads to a graded-connected Hopf algebra structure, and there is a universal recipe for constructing quasisymmetric functions out of this data.
One of the observations we make in this paper is that Shareshian and Wachs's $\csf_q$ can be constructed using this recipe, so that it is uniquely determined by a very small amount of data.
This data consists of a single coefficient in $\CC(q)$ for each ordered graph, rather than a whole quasisymmetric function over $\CC(q)$ for each ordered graph.
Furthermore, it turns out that each of these coefficients is either zero or a power of $q$.
This can be thought of as both a quantitative and a qualitative reduction in the work needed to show that the $\csf_q$ is what pops out of the construction involving Hessenberg varieties.

The Hopf algebra structure suggests that Hessenberg varieties are recursively structured in a very principled way.
Guided by this, we show that the Hessenberg construction respects the multiplicative and comultiplicative structures present in the Hopf algebra on Dyck paths by giving explicit decompositions of the equivariant cohomology rings.
The base case involves identifying the subspace of these equivariant cohomology rings on which the symmetric group acts according to the sign representation, which we also do explicitly.

\subsection{Acknowledgments}
The starting point for this paper, the idea that there is a Hopf algebra of Dyck paths which is applicable to $q$-chromatic quasisymmetric functions, really comes from a suggestive formula in an interesting short note by Athanasiadis \cite{athanasiadis15} (see \autoref{rem:athanasiadis}).
I would also like to thank Amy Pang, Franco Saliola, Hugh Thomas and Nathan Williams for encouragement and very helpful discussions on the topic, and LaCIM for providing funding and a wonderful work environment.

\section{The \texorpdfstring{$q$}{q}-chromatic quasisymmetric function}\label{sec:qcsf}
Given a (finite simple undirected) graph $G = (V, E)$ and a set of colours $C$, a \emph{colouring} is just a function $\kappa \colon V \to C$ assigning a colour to each vertex of $G$.
We are typically interested in \emph{proper} colourings, that is, ones which assign different colours to neighbouring vertices.

The \emph{chromatic polynomial} of $G$ is the function
\begin{equation}
  P(G, r) = \sum_{\substack{\kappa \colon V \to [r] \\ \text{proper}}} 1,
\end{equation}
which counts the number of proper colourings of $G$ with $r$ colours.
Here, we take the set of $r$ colours to be $[r] = \set{1, 2, \ldots, r}$;
since a colouring of $G$ stays proper if the colours are relabelled, this is no loss of generality.

We can get a more refined count of the proper colourings of $G$ by considering the number of times each colour is used.
Stanley's \emph{chromatic symmetric function} \cite{stanley95}, which keeps track of this information, is the formal power series
\begin{equation}
  \csf(G, \xx) = \sum_{\substack{\kappa \colon V \to \PP \\ \text{proper}}} \xx_\kappa.
\end{equation}
Here, the set of colours is the countable set $\PP = \set{1, 2, \ldots}$, with a corresponding set of indeterminates $\xx = (x_1, x_2, \ldots)$, and the monomial $\xx_\kappa$ is the product of $x_{\kappa(v)}$ over all vertices $v \in V$, so that the exponent of $x_i$ in $\xx_\kappa$ counts the number of times the colour $i$ is used in the colouring $\kappa$.
Relabelling the colours in each colouring amounts to permuting the indeterminates, so that $\csf(G, \xx)$ is indeed a symmetric function, invariant under permutations of $\xx$.
The chromatic polynomial can be recovered from the chromatic symmetric function as
\begin{equation}
  P(G, r) = \csf\paren[\big]{G, (\underbrace{1, \ldots, 1}_{\text{$r$ copies}}, 0, 0, 0, \ldots)}.
\end{equation}

If the set of vertices and the set of colours each come equipped with a total order, then we can further refine the count of proper colourings by counting \emph{ascents} in each colouring $\kappa$: pairs of vertices $u, v$ such that
\begin{itemize}[nosep,label=--]
  \item $u$ and $v$ are neighbours,
  \item $u < v$ in the ordering on vertices, and
  \item $\kappa(u) < \kappa(v)$ in the ordering on colours.
\end{itemize}
Shareshian and Wachs's \emph{$q$-chromatic quasisymmetric function} \cite{shareshian-wachs12}, which takes this into account, is the formal power series
\begin{equation}\label{eq:def-qcsf}
  \csf_q(G, \xx) = \sum_{\substack{\kappa \colon V \to \PP \\ \text{proper}}} q^{(\text{\# ascents of $\kappa$ on $G$})} \, \xx_\kappa,
\end{equation}
where now $G = (V, E, {<})$ is an \emph{ordered} graph, and we take the usual ordering on $\PP$.
The number of ascents of a colouring is only invariant under relabellings of the colours which preserve the order of the colours used, so $\csf_q(G, \xx)$ is only a \emph{quasi}symmetric function in general.
However, the graphs we consider in this paper (which come from natural unit interval orders) have the property that $\csf_q(G, \xx)$ is in fact symmetric.
In any case, setting $q=1$ recovers the usual chromatic symmetric function.

\section{The Hopf algebra of ordered graphs, with a twist}\label{sec:graph}
Our goal in this section is to define a Hopf algebra for ordered graphs.
Some of the operations could be naturally defined over a smaller ring, but for simplicity we take the field $\CC(q)$ of rational functions in $q$ as ground ring. Our ground set $\og$ is the set of all $\CC(q)$-linear combinations of ordered graphs (up to isomorphism).
We will define the $r$-fold multiplication map
\begin{equation}
  \nabla_r \colon \underbrace{\og \otimes \cdots \otimes \og}_{\text{$r$ copies}} \to \og
\end{equation}
on the basis of $r$-fold tensor products of ordered graphs by
\begin{equation}\label{eq:def-mult}
  \nabla_r(G_1 \otimes \cdots \otimes G_r) = G_1 \olex \cdots \olex G_r,
\end{equation}
where $\olex$ is essentially the concatenation operation on ordered graphs (see below).
The $r$-fold comultiplication map
\begin{equation}
  \Delta_r \colon \og \to \underbrace{\og \otimes \cdots \otimes \og}_{\text{$r$ copies}}
\end{equation}
will be given on the basis of ordered graphs as a sum over all colourings of the ordered graph under consideration, where in addition we keep track of a statistic, the number of ascents of each colouring (see below again):
\begin{equation}\label{eq:def-comult}
  \Delta_r(G) = \sum_{\substack{\kappa \colon V \to [r] \\ \text{arbitrary}}} q^{(\text{\# ascents of $\kappa$ on $G$})} \, G|_\kappa.
\end{equation}
This is analogous to a sum over all deshufflings of $G$.
If it weren't for the power of $q$ keeping track of the number of ascents, this would fit the framework of Hopf monoids of Aguiar and Mahajan \cite{aguiar-mahajan12}*{Section 9.4}, and would in fact be a well-known Hopf algebra related to Stanley's $\csf$ (see \cite{aguiar-bergeron-sottile06}*{Example 4.5}).

To make $\og$ graded-connected, which is sufficient to transform it from a bialgebra to a Hopf algebra, we take the degree of an ordered graph to be its number of vertices.
The subsections below give the details of these constructions, and the facts needed to show that $\og$ is indeed a graded-connected Hopf algebra.
Feel free to skip them.
For ease of citation, the summary is that:
\begin{proposition}
  With the definitions of this section, the space $\og$ is a graded-connected Hopf algebra over $\CC(q)$.
\end{proposition}

\subsection{Isomorphism}
Throughout this paper, we only care about ordered graphs up to isomorphism. Two ordered graphs $G = (V, E, {<})$ and $G' = (V', E', {<'})$ are \emph{isomorphic} is there is a bijection $\varphi \colon V \to V'$ such that:
\begin{itemize}[nosep,label=--]
  \item $\set{u, v}$ is an edge of $G$ iff $\set{\varphi(u), \varphi(v)}$ is an edge of $G'$, and
  \item $u < v$ in the order on $G$ iff $\varphi(u) <' \varphi(v)$ in the order on $G'$.
\end{itemize}
The isomorphism class of $G = (V, E, {<})$ contains a unique graph with the ordered vertex set $\set{1 < 2 < \cdots < n}$, where $n = \abs{V}$, which can be thought of as a canonical representative for $G$. However, for the constructions below, we keep the flexibility of using other vertex sets and other orderings.

\subsection{Colourings, ascents, descents}
Let $G = (V, E, {<})$ be an ordered graph and let $\kappa \colon V \to [r]$ be a colouring, that is, an arbitrary function.
Let $e = \set{u, v} \in E$ be an edge, with $u < v$ according to the ordering.
Then, the edge $e$ can relate to the colouring $\kappa$ in three different ways:
\begin{itemize}[nosep,label=--]
  \item if $\kappa(u) = \kappa(v)$, then $e$ is \emph{monochromatic};
  \item if $\kappa(u) < \kappa(v)$, then $e$ is an \emph{ascent};
  \item if $\kappa(u) > \kappa(v)$, then $e$ is a \emph{descent}.
\end{itemize}
Note that $\kappa$ is a proper colouring iff there are no monochromatic edges.

\subsection{Restriction}
Let $U \subseteq V$ be a subset of the vertices of an ordered graph $G = (V, E, {<})$.
Then, we define the \emph{restriction} $G|_U = (U, E|_U, {<}|_U)$ of $G$ to $U$ in the obvious way:
$E|_U$ is the set of edges with both endpoints in $U$, and ${<}|_U$ is the restriction of the order relation on $V$ to $U$.
Given a colouring $\kappa \colon V \to [r]$, we also define the \emph{restriction} $G|_\kappa$ to be the list of ordered graphs
\begin{equation}
  (G|_{\kappa^{-1}(1)},\; G|_{\kappa^{-1}(2)},\; \ldots,\; G|_{\kappa^{-1}(r)}),
\end{equation}
that is, the restriction of $G$ to each of the colour classes, in order.
Note that the edges which survive in the restriction are exactly the monochromatic edges; the ascents and descents disappear.

\subsection{Lexicographic union}
Conversely, given a list of ordered graphs
\begin{equation}
  (G_1,\; G_2,\; \ldots,\; G_r),
\end{equation}
we can construct a particular ordered graph $G = (V, E, {<})$ that restricts to it, which we call the \emph{lexicographic union} $\mathop{\olex}_i G_i$ of this list. If the listed graphs are $G_i = (V_i, E_i, {<_i})$, then:
\begin{itemize}[nosep,label=--]
  \item $V$ is the disjoint union of $V_1, \ldots, V_r$;
  \item $E$ is the disjoint union of $E_1, \ldots, E_r$; and
  \item $u < v$ iff either $u <_i v$ in some $V_i$, or $u \in V_i$ and $v \in V_j$ with $i < j$.
\end{itemize}
If we colour the vertices of $V_i \subseteq V$ with colour $i$, then every edge is monochromatic, and the restriction of $G$ to this colouring is $(G_1, G_2, \ldots, G_r)$.
We will also use the notation $\mathop{\olex}_i V_i$ to refer to the lexicographic union $(V, <)$ of ordered sets $(V_i, {<_i})$, where we ignore edge sets completely.

\subsection{Reshuffling}\label{sec:reshuffling}
Now, consider a list of $r$ ordered graphs
\begin{equation}
  (G_1,\; G_2,\; \ldots,\; G_r)
\end{equation}
as above, with $G_i = (V_i, E_i, {<_i})$, so that we can talk of the lexicographic union $G_1 \olex \cdots \olex G_r$, and let $\kappa \colon V_1 \olex \cdots \olex V_r \to [s]$ be an arbitrary colouring. In this situation, we can construct the restriction
\begin{equation}
    \bigolex_{i=1,\ldots,r} G_i \bigg|_\kappa = (G'_1,\; G'_2,\; \ldots,\; G'_s).
\end{equation}
The compatibility axiom for the multiplication and the comultiplication of the Hopf algebra $\og$ essentially says that we get this second list of $s$ ordered graphs $G'_i = (V'_i, E'_i, {<'_i})$ whether we first construct the lexicographic union and then the restriction, or the restriction first and the lexicographic union second.
(The compatibility axiom also says that the total number of ascents should be the same in both cases.)
To show that this is the case, let us be a bit more explicit about what we mean by doing ``the restriction first and the lexicographic union second''.

Restricting the domain of the colouring $\kappa$ on $V_1 \olex \cdots \olex V_r$ gives a colouring $\kappa_i = \kappa|_{V_i} \colon V_i \to [s]$ for each $i = 1, \ldots, r$, so we can construct $r$ lists of $s$ ordered graphs each,
\begin{equation}\label{eq:rtable}\begin{aligned}
  G_1|_{\kappa_1} &= (G''_{1,1},\; G''_{1,2},\; \ldots,\; G''_{1,s}), \\
  G_2|_{\kappa_2} &= (G''_{2,1},\; G''_{2,2},\; \ldots,\; G''_{2,s}), \\
  &\vdotswithin{=} \\
  G_r|_{\kappa_r} &= (G''_{r,1},\; G''_{r,2},\; \ldots,\; G''_{r,s}),
\end{aligned}\end{equation}
where $G''_{i,j}$ is the restriction of $G_i$ to the vertices of $V_i$ which are coloured $j$ by $\kappa$.
There is only one sensible way to reassemble these ordered graphs into a list of $s$ ordered graphs, and one can check that
\begin{equation}\label{eq:utable}\begin{aligned}
  G''_{1,1} \olex G''_{2,1} \olex \cdots \olex G''_{r,1} &= G'_1, \\
  G''_{1,2} \olex G''_{2,2} \olex \cdots \olex G''_{r,2} &= G'_2, \\
  &\vdotswithin{=} \\
  G''_{1,s} \olex G''_{2,s} \olex \cdots \olex G''_{r,s} &= G'_s
\end{aligned}\end{equation}
by unrolling the definitions of restriction to a colouring and lexicographic union.
The only subtlety is that the table of intermediate ordered graphs $G''_{i,j}$ is transposed between \eqref{eq:rtable} and \eqref{eq:utable}.

As for ascents, note that $G_1 \olex \cdots \olex G_r$ does not have any edges between $V_i$ and $V_j$ for $i \neq j$. Thus, we have
\begin{multline}
  (\text{\# ascents of $\kappa$ on $G_1 \olex \cdots \olex G_r$}) \\
   = (\text{\# ascents of $\kappa_1$ on $G_1$})
   + \cdots
   + (\text{\# ascents of $\kappa_r$ on $G_r$}).
\end{multline}

\subsection{Associativity}
Let us show that the multiplication map from \eqref{eq:def-mult} turns $\og$ into an associative unital $\CC(q)$-algebra.
We have a basis of
\begin{equation}
  \underbrace{\og \otimes \cdots \otimes \og}_{\text{$r$ copies}}
\end{equation}
which consists of all lists of $r$ ordered graphs, of the form
\begin{equation}
  G_1 \otimes \cdots \otimes G_r.
\end{equation}
On this basis element, the $r$-fold multiplication map is defined to be
\begin{equation}
  \nabla_r(G_1 \otimes \cdots \otimes G_r) = G_1 \olex \cdots \olex G_r \in \og.
\end{equation}
Now, consider a bracketing of the list of $r$ ordered graphs, of the form
\begin{equation}\label{eq:bracket}\begin{aligned}
  &\phantom{{}\otimes{}} (G_1 \otimes \cdots \otimes G_{r_1}) \\
  &\otimes (G_{r_1+1} \otimes \cdots \otimes G_{r_1+r_2})\\
  &\vdotswithin{\otimes} \\
  &\otimes (G_{r_1+\cdots+r_{s-1}+1} \otimes \cdots \otimes G_{r_1+\cdots+r_{s-1}+r_s}),
\end{aligned}\end{equation}
where there are $s$ brackets, the $i$th bracket contains $r_i$ ordered graphs, and $r_1 + \cdots + r_s = r$.
To have associativity, we need the trivial condition that $\nabla_1 \colon \og \to \og$ be the identity map, and the condition that the multiplication done according to the bracketing above,
\begin{equation}
  \nabla_s \circ (\nabla_{r_1} \otimes \nabla_{r_2} \otimes \cdots \otimes \nabla_{r_s}),
\end{equation}
be equal to the multiplication $\nabla_r$.
Given the definition of lexicographic union, this boils down to verifying that the bijection of ordered sets
\begin{equation}
  \ell \colon [r_1] \olex [r_2] \olex \cdots \olex [r_s] \to [r]
\end{equation}
given on each $[r_j] \subseteq [r_1] \olex [r_2] \olex \cdots \olex [r_s]$ by
\begin{equation}\label{eq:ell}\begin{aligned}
  \ell|_{[r_j]} \colon [r_j] &\to [r] \\
  i \;\; &\mapsto \; r_1 + \cdots + r_{j-1} + i
\end{aligned}\end{equation}
respects the order, which is immediate.
It follows that the empty ordered graph $G_0 = (\varnothing, \varnothing, {<})$, which is the image of the empty list of ordered graphs under the 0-fold multiplication map $\nabla_0 \colon \CC(q) \to \og$, is a unit element for $\og$.

\subsection{Coassociativity}
Dually, let us show that $\og$ is a coassociative counital coalgebra under the comultiplication map from \eqref{eq:def-comult}.
Recall this is the map
\begin{equation}
  \Delta_r \colon \og \to \underbrace{\og \otimes \cdots \otimes \og}_{\text{$r$ copies}}
\end{equation}
defined on the basis of $\og$ which consists of all ordered graphs by
\begin{equation}\label{eq:redef-comult}
  \Delta_r(G) = \sum_{\substack{\kappa \colon V \to [r] \\ \text{arbitrary}}} q^{(\text{\# ascents of $\kappa$ on $G$})} \,
  G|_{V_1} \otimes \cdots \otimes G|_{V_r},
\end{equation}
where $V_i = \kappa^{-1}(i)$ is the set of vertices of $G$ assigned to colour $i$ by $\kappa$ for each $i = 1, \ldots, r$.
As with associativity, there is the trivial condition that $\Delta_1 \colon \og \to \og$ be the identity map, and the condition that
\begin{equation}\label{eq:coassoc}
  (\Delta_{r_1} \otimes \Delta_{r_2} \otimes \cdots \otimes \Delta_{r_s}) \circ \Delta_s = \Delta_r
\end{equation}
for a bracketing as in \eqref{eq:bracket}, which we can establish by showing that each term in the sum on the left-hand side of \eqref{eq:coassoc} corresponds to a term in the sum on the right-hand side and vice versa (after substituting the defining sum \eqref{eq:redef-comult} and distributing the sums over the tensor product).

Fix a bracketing, and hence an order-preserving bijection
\begin{equation}
  \ell \colon [r_1] \olex [r_2] \olex \cdots \olex [r_s] \to [r]
\end{equation}
as defined in \eqref{eq:ell}. For each colouring $\kappa \colon V \to [r]$ on $G$, we can define a coarser colouring $\kappa' \colon V \to [s]$ by the composition
\begin{equation}
V \overset{\kappa}{\longrightarrow} [r] \overset{\ell^{-1}}{\longrightarrow} [r_1] \olex \cdots \olex [r_s] \longrightarrow [s],
\end{equation}
where the last arrow is the map sending every colour $i$ in the $j$th summand $[r_j]$ to $j$. This colouring is coarser in the following sense: if two vertices $u, v \in V$ are assigned the same colour in $[r]$ by $\kappa$, then they are assigned the same colour in $[s]$ by $\kappa'$, so that the colour classes for $\kappa'$ are obtained by clumping together colour classes for $\kappa$. Let $V'_j = \kappa'^{-1}(j)$ be the $j$th colour class of $\kappa'$. To recover $\kappa$ from $\kappa'$, the extra information needed is exactly given by the colourings
\begin{equation}
  \kappa'_j = (\ell^{-1} \circ \kappa)\big|_{V'_j} \colon V'_j \to [r_j], \qquad j = 1, \ldots, s.
\end{equation}
This establishes a correspondence between the terms on the right-hand side of \eqref{eq:coassoc}, each of which is given by:
\begin{itemize}[nosep,label=--]
  \item an arbitrary colouring $\kappa \colon V \to [r]$,
\end{itemize}
and the terms on the left-hand side, each of which is given by:
\begin{itemize}[nosep,label=--]
  \item an arbitrary colouring $\kappa' \colon V \to [s]$, and
  \item an arbitrary colouring $\kappa'_j \colon V'_j \to [r_j]$ for each $j = 1, \ldots, s$.
\end{itemize}
It remains to show that corresponding terms are, in fact, equal.
For the right-hand term, the $i$th ordered graph in the tensor product is the restriction of $G$ to the colour class $V_i = \kappa^{-1}(i)$ of $\kappa$.
For the left-hand term, the $i$th factor appears in the $j$th bracket, where $V'_j$ is the colour class of $\kappa'$ containing $V_i$, and it is the restriction of $G$ first to $V'_j$, then to $V_i$; this is the same ordered graph.
For the right-hand term, the power of $q$ appearing in the coefficient is
\begin{equation}\label{eq:rhs-ascents}
  (\text{\# ascents of $\kappa$ on $G$}).
\end{equation}
For the left-hand term, the power of $q$ in the coefficient is
\begin{multline}\label{eq:lhs-ascents}
  (\text{\# ascents of $\kappa'$ on $G$}) \\
  {} + (\text{\# ascents of $\kappa'_1$ on $G|_{V'_1}$})
  + \cdots + (\text{\# ascents of $\kappa'_s$ on $G|_{V'_s}$}).
\end{multline}
To see that these two numbers are the same, consider an edge $\set{u, v} \in E$ with $u < v$. This edge is monochromatic for $\kappa$, that is, $\kappa(u) = \kappa(v)$ when:
\begin{itemize}[nosep,label=--]
  \item $u, v \in V'_j$ for some $j$, so that $\kappa'(u) = \kappa'(v) = j$, and $\kappa'_j(u) = \kappa'_j(v)$; that is, the edge is monochromatic for $\kappa'$ and $\kappa'_j$.
\end{itemize}
The edge is an ascent of $\kappa$, that is, $\kappa(u) < \kappa(v)$ when either:
\begin{itemize}[nosep,label=--]
  \item $u \in V'_i$ and $v \in V'_j$ with $i < j$, in which case it is an ascent of $\kappa'$; or
  \item $u, v \in V'_j$ for some $j$ and $\kappa'_j(u) < \kappa'_j(v)$, in which case it is monochromatic for $\kappa'$ and an ascent of $\kappa'_j$.
\end{itemize}
The edge is an descent of $\kappa$, that is, $\kappa(u) > \kappa(v)$ when either:
\begin{itemize}[nosep,label=--]
  \item $u \in V'_i$ and $v \in V'_j$ with $i > j$, in which case it is an descent of $\kappa'$; or
  \item $u, v \in V'_j$ for some $j$ and $\kappa'_j(u) > \kappa'_j(v)$, in which case it is monochromatic for $\kappa'$ and an descent of $\kappa'_j$.
\end{itemize}
This breakdown of the cases shows that \eqref{eq:rhs-ascents} and \eqref{eq:lhs-ascents} count exactly the same set of edges of $G$ as ascents.
Thus, corresponding terms are equal, and the comultiplication maps are coassociative.

The only ordered graph which has a colouring using no colours is the empty ordered graph $G_0 = (\varnothing, \varnothing, {<})$.
Thus, the 0-fold comultiplication map $\Delta_0 \colon \og \to \CC(q)$, which is the counit, is the map which extracts the coefficient of $G_0$ in the basis of ordered graphs.

\subsection{Compatibility}
The $r$-fold multiplication maps
\begin{equation}
  \nabla_r \colon \underbrace{\og \otimes \cdots \otimes \og}_{\text{$r$ copies}} \to \og
\end{equation}
turn $\og$ into an associative unital algebra. By extension, the $s$-fold tensor power
\begin{equation}
  \underbrace{\og \otimes \cdots \otimes \og}_{\text{$s$ copies}}
\end{equation}
is also an associative unital algebra under component-wise multiplication.
Dually, the $s$-fold multiplication maps
\begin{equation}
  \Delta_s \colon \og \to \underbrace{\og \otimes \cdots \otimes \og}_{\text{$s$ copies}}
\end{equation}
turn $\og$ into a coassociative counital coalgebra, and under component-wise comultiplication, so is the $r$-fold tensor power
\begin{equation}
  \underbrace{\og \otimes \cdots \otimes \og}_{\text{$r$ copies}}.
\end{equation}
The compatibility axiom for $\og$ to be a bialgebra is that the multiplication maps be coalgebra maps, or equivalently, that the comultiplication maps be algebra maps.
In other words, the two natural ways to define a map
\begin{equation}
  \underbrace{\og \otimes \cdots \otimes \og}_{\text{$r$ copies}}
  \longrightarrow
  \underbrace{\og \otimes \cdots \otimes \og}_{\text{$s$ copies}},
\end{equation}
namely, a multiplication followed by a comultiplication, or a component-wise comultiplication followed by a component-wise multiplication, must agree.
The facts needed to check that this is the case are given in the section on `reshuffling' (\autoref{sec:reshuffling}).

\subsection{Graded-connectedness}
For $\og$ to be graded-connected, we need a decomposition of it of the form
\begin{equation}
  \og = \og_0 \oplus \og_1 \oplus \og_2 \oplus \og_3 \oplus \cdots,
\end{equation}
where we call the space $\og_n$ the homogeneous component of degree $n$.
We get a corresponding decomposition of the space
\begin{equation}
  \underbrace{\og \otimes \og \otimes \cdots \otimes \og}_{\text{$r$ copies}},
\end{equation}
by defining the homogeneous component of degree $n$ to be
\begin{equation}
  \bigoplus (\og_{n_1} \otimes \og_{n_2} \otimes \cdots \otimes \og_{n_r}),
\end{equation}
where the sum is over all sequences of $r$ natural numbers $n_1, n_2, \ldots, n_r$ such that $n_1 + n_2 + \cdots + n_r = n$.
We require that this decomposition be compatible with the multiplicative and the comultiplicative structure, in the sense that
\begin{equation}
  \nabla_r\paren[\Big]{\bigoplus (\og_{n_1} \otimes \cdots \otimes \og_{n_r})}
  \subseteq
  \og_n
\end{equation}
and
\begin{equation}
  \Delta_r(\og_n)
  \subseteq
  \bigoplus (\og_{n_1} \otimes \cdots \otimes \og_{n_r}).
\end{equation}
If this is the case, then $\og$ is graded. For it to be graded-connected, we also need $\og_0$ to be 1-dimensional. If this is the case, then the 0-fold multiplication map $\nabla_0 \colon \CC(q) \to \og_0$ and the 0-fold comultiplication map $\Delta_0 \colon \og_0 \to \CC(q)$, which can be restricted to $\og_0$, are inverses of one another, so that they are linear isomorphisms.

If $\og$ is a graded-connected bialgebra, then it is automatically a Hopf algebra, meaning that there is an antipode map
\begin{equation}
  S \colon \og \to \og,
\end{equation}
as given by, for example, Takeuchi's formula (see \cite{aguiar-mahajan12}*{Section 5}).

In fact, it is easy to check that $\og$ is graded-connected if we take $\og_n$ to be the subspace of $\og$ spanned by the ordered graphs which have $n$ vertices, as this is compatible with the multiplication and comultiplication maps, and there is only one ordered graph with no vertices, namely $G_0 = (\varnothing, \varnothing, {<})$.

\section{A few facts about quasisymmetric functions}\label{sec:qsym}
As alluded to in the \hyperref[sec:intro]{introduction}, Aguiar, Bergeron and Sottile \cite{aguiar-bergeron-sottile06} have shown that the space $\qsym$ of quasisymmetric functions is a universal object for graded-connected Hopf algebras, in a way that lets us easily construct and characterize Hopf-algebraic maps to $\qsym$. In this section, we recall just enough details about $\qsym$ to state this precisely.

As in \autoref{sec:graph}, we take our ground ring to be the field $\CC(q)$.

A \emph{quasisymmetric function} is a formal power series of bounded degree in the countable ordered set of indeterminates $\xx = (x_1, x_2, \ldots)$ with the following invariance property: any two monomials with the same ordered list of nonzero exponents must have the same coefficient.
In other words, the quasisymmetric functions have a basis given by the monomial quasisymmetric functions
\begin{equation}
  M_\alpha = \sum_{i_1 < i_2 < \cdots < i_r} x_{i_1}^{\alpha_1} x_{i_2}^{\alpha_2} \cdots x_{i_r}^{\alpha_r},
\end{equation}
where $\alpha = (\alpha_1, \alpha_2, \ldots, \alpha_r)$ is any finite list of positive integers.
In fact, the space $\qsym$ of quasisymmetric functions is a graded-connected Hopf algebra, where the multiplication maps
\begin{equation}
  \nabla_r \colon \underbrace{\qsym \otimes \cdots \otimes \qsym}_{\text{$r$ copies}} \to \qsym
\end{equation}
and the grading are inherited from the algebra of power series.
We will not need any details about the comultiplication maps
\begin{equation}
  \Delta_r \colon \qsym \to \underbrace{\qsym \otimes \cdots \otimes \qsym}_{\text{$r$ copies}}
\end{equation}
and the antipode
\begin{equation}
  S \colon \qsym \to \qsym,
\end{equation}
other than the fact that they exist.
The \emph{canonical character} on $\qsym$ is the multiplicative linear functional
\begin{equation}
  \zq \colon \qsym \to \CC(q)
\end{equation}
which evaluates the indeterminates $\xx$ at $\xx = (1, 0, 0, 0, \ldots)$, so that
\begin{equation}
  \zq(M_\alpha) = \begin{cases}
    1 &\text{if $\alpha$ has at most one part,} \\
    0 &\text{if $\alpha$ has at least two parts.}
  \end{cases}
\end{equation}
With these definitions in place, we have the following universality result:
\begin{theorem}[\cite{aguiar-bergeron-sottile06}*{Theorem 4.1}]\label{thm:universal}
  For each pair $(\HH, \zeta)$ where $\HH$ is a graded-connected Hopf algebra and $\zeta$ is a multiplicative function from $\HH$ to the ground ring, there exists a unique map of graded Hopf algebras
  \begin{equation}
    \Psi_\zeta \colon \HH \to \qsym
  \end{equation}
  which sends $\zeta$ to $\zq$. Moreover, the coefficient of $M_\alpha$ in $\Psi_\zeta(h)$ is given by
  \begin{equation}
    (\underbrace{\zeta \otimes \zeta \otimes \cdots \otimes \zeta}_{\text{$r$ copies}}) \circ
    (\pi_{\alpha_1} \otimes \pi_{\alpha_2} \otimes \cdots \otimes \pi_{\alpha_r}) \circ
    \Delta_r(h),
  \end{equation}
  where $\alpha = (\alpha_1, \alpha_2, \ldots, \alpha_r)$ is a list of $r$ positive integers, $\Delta_r$ is the $r$-fold comultiplication map of $\HH$, and $\pi_n$ is the projection onto the homogeneous component of degree $n$ of $\HH$.
\end{theorem}

\section{The \texorpdfstring{$q$}{q}-chromatic quasisymmetric function, revisited}\label{sec:qcsf2}
Now, let us use the recipe of \autoref{sec:qsym} to reconstruct the Shareshian--Wachs $\csf_q$ as a map of graded Hopf algebras between $\og$ and $\qsym$.
We will also introduce two variants: the strict and the weak chromatic quasisymmetric functions.

Consider the three characters $\zeta_0, \zeta_1, \zeta_q \colon \og \to \CC(q)$ defined on the basis of ordered graphs by
\begin{align}
  \zeta_0(G) &= \begin{cases}
    1 &\text{if $G$ has no edges,} \\
    0 &\text{otherwise,}
  \end{cases} \\
  \zeta_1(G) &= 1, \\
  \zeta_q(G) &= q^{(\text{\# edges of $G$})}.
\end{align}
Since the multiplication of ordered graphs is a disjoint union on sets of edges, these functions are multiplicative.
By \autoref{thm:universal}, there are corresponding maps $\Psi_0, \Psi_1, \Psi_q \colon \og \to \qsym$.
\begin{theorem}\label{thm:psi0}
  For every ordered graph $G$, we have $\Psi_0(G) = \csf_q(G)$.
\end{theorem}
\begin{proof}
  This follows by unrolling the definition of $\Psi_0(G)$.

  Let $\alpha = (\alpha_1, \alpha_2, \ldots, \alpha_r)$ be a list of $r$ positive integers.
  Then, the coefficient of $M_\alpha$ in $\Psi_0(G)$ is
  \begin{equation}
    (\underbrace{\zeta_0 \otimes \cdots \otimes \zeta_0}_{\text{$r$ copies}}) \circ
    (\pi_{\alpha_1} \otimes \cdots \otimes \pi_{\alpha_r}) \circ
    \Delta_r(G).
  \end{equation}
  By definition, the $r$-fold comultiplication of $G$ is
  \begin{equation}
    \Delta_r(G) = \sum_{\substack{\kappa \colon V \to [r] \\ \text{arbitrary}}} q^{(\text{\# ascents of $\kappa$ on $G$})} \,
    G|_{V_1} \otimes \cdots \otimes G|_{V_r},
  \end{equation}
  where $V_i = \kappa^{-1}(i)$ is the set of vertices assigned colour $i$.
  We can substitute this in the previous expression to get
  \begin{equation}
    \sum_{\substack{\kappa \colon V \to [r] \\ \text{arbitrary}}}
    q^{(\text{\# ascents of $\kappa$ on $G$})} \,
    \paren[\big]{\zeta_0 \circ \pi_{\alpha_1}(G|_{V_1})}
    \otimes \cdots \otimes
    \paren[\big]{\zeta_0 \circ \pi_{\alpha_1}(G|_{V_r})}.
  \end{equation}
  By the definition of the projection $\pi_n$, the summand is zero unless each $V_i$ has exactly $\alpha_i$ vertices, so that $\kappa$ is a colouring where colour $i$ is used $\alpha_i$ times.
  Conversely, $\pi_{\alpha_i}(G|_{V_i}) = G|_{V_i}$ for each $i$ if $\kappa$ is such a colouring.
  By the definition of $\zeta_0$, the summand is zero if $\kappa$ has any monochromatic edges. Conversely, if $\kappa$ is a proper colouring, then $\zeta_0(G|_{V_i}) = 1$ for each $i$.
  Combining these two facts, we can rewrite the coefficient of $M_\alpha$ in $\Psi_0(G)$ as
  \begin{equation}
    \sum_{\substack{\kappa \colon V \to [r] \\ \text{proper} \\ \abs{V_i} = \alpha_i}}
    q^{(\text{\# ascents of $\kappa$ on $G$})}
  \end{equation}
  Now, recall that the monomial quasisymmetric function $M_\alpha$ is defined by
  \begin{equation}
    M_\alpha = \sum_{j_1 < \cdots < j_r} x_{j_1}^{\alpha_1} \cdots x_{j_r}^{\alpha_r}.
  \end{equation}
  A proper colouring $\kappa' \colon V \to \PP$ which uses colours $j_1 < \cdots < j_r$, and uses them $\alpha_1, \ldots, \alpha_r$ times respectively, can be uniquely factored as a composition
  \begin{equation}
    V \overset{\kappa}{\longrightarrow} [r] \hookrightarrow \PP,
  \end{equation}
  where $\kappa$ a proper colouring where colour $i$ is used $\alpha_i$ times, and the injection on the right is order preserving.
  Thus, we can write $\Psi_0(G)$ as
  \begin{equation}
    \Psi_0(G) = \sum_{\substack{\kappa \colon V \to \PP \\ \text{proper}}} q^{(\text{\# ascents of $\kappa$ on $G$})} \, \xx_\kappa,
  \end{equation}
  which is exactly the definition \eqref{eq:def-qcsf} of $\csf_q(G, \xx)$.
\end{proof}

\begin{remark}\label{rem:athanasiadis}
  This approach to the $\csf_q$, and indeed the very definition of the Hopf algebra $\og$, are directly inspired by a formula proven by Athanasiadis \cite{athanasiadis15}*{Equation 17}, which can be interpreted as a proof of the equation
  \begin{equation}
    \Delta_r\paren[\big]{\csf_q(G)} = \csf_q\paren[\big]{\Delta_r(G)}
  \end{equation}
  in the power-sum basis of symmetric functions. The fact that $\csf_q$ is multiplicative was already well-known, and this suggestion that it is also comultiplicative prompted the search for an underlying Hopf algebra.
\end{remark}

In a proper colouring, every edge is either an ascent or a descent, never monochromatic, so there is no reason to distinguish between strict and weak inequalities.
However, for arbitrary colourings, there are two sensible generalizations of the notion of `ascent'.
So far we have taken the convention that an ascent of a colouring $\kappa$ is a \emph{strict} ascent, that is, an edge $\set{u, v}$ with $u < v$ and $\kappa(u) < \kappa(v)$.
We could also have considered \emph{weak} ascents, that is, an edge $\set{u, v}$ with $u < v$ and $\kappa(u) \leq \kappa(v)$.
Given these definitions, we have the following interpretations of the quasisymmetric functions $\Psi_1(G)$ and $\Psi_q(G)$ for an ordered graph $G$, which can be proved in much the same way as \autoref{thm:psi0}:
\begin{align}
  \Psi_1(G) = \sum_{\substack{\kappa \colon V \to \PP \\ \text{arbitrary}}} q^{(\text{\# strict ascents of $\kappa$ on $G$})} \, \xx_\kappa, \\
  \Psi_q(G) = \sum_{\substack{\kappa \colon V \to \PP \\ \text{arbitrary}}} q^{(\text{\# weak ascents of $\kappa$ on $G$})} \, \xx_\kappa.
\end{align}
Note that in both cases, we are allowing arbitrary colourings, rather than just proper colourings. We will call these the \emph{strict chromatic quasisymmetric function} and the \emph{weak chromatic quasisymmetric function}, respectively.
More generally, for any $t \in \CC(q)$, we could consider the multiplicative character
\begin{equation}
  \zeta_t(G) = t^{(\text{\# edges of $G$})},
\end{equation}
and the associated morphism $\Psi_t$ would have the interpretation
\begin{equation}
  \Psi_t(G) = \sum_{\substack{\kappa \colon V \to \PP \\ \text{arbitrary}}} t^{(\text{\# monochromatic edges of $\kappa$ on $G$})} q^{(\text{\# ascents of $\kappa$ on $G$})} \, \xx_\kappa.
\end{equation}

\section{The subalgebra of Dyck paths}\label{sec:dyck}
In this section, we introduce a subclass of ordered graphs which is closed under lexicographic union and restriction, so that they span a Hopf subalgebra $\dyck$ of $\og$. For the rest of the paper we will focus on $\dyck$ rather than $\og$, since Hessenberg varieties are only defined for these ordered graphs.

Fix $n \geq 0$.
A \emph{Hessenberg function} is a function $h \colon [n] \to [n]$ which is:
\begin{itemize}[nosep,label=--]
  \item \emph{extensive}, meaning that $i \leq h(i)$ for $i = 1, \ldots, n$, and
  \item \emph{increasing}, meaning that $h(i) \leq h(i+1)$ for $i = 1, \ldots, n-1$.
\end{itemize}
Given such a function, there is an associated poset $P(h)$ on $[n]$, defined by
\begin{equation}
  i <_h j
  \quad \Longleftrightarrow \quad
  h(i) < j.
\end{equation}
There is also an associated ordered graph $G(h)$ on $[n]$, where the total order on $[n]$ is the usual numeric order, and the edges are given by
\begin{equation}
  \set{i, j} \in E(h)
  \quad \Longleftrightarrow \quad
  i < j \leq h(i).
\end{equation}
The class of posets of the form $P(h)$ is the class of unit interval orders (see \cite{shareshian-wachs14}*{Proposition 4.1}), that is, posets whose elements can be modelled as unit intervals on the real line, where an interval is less than another if it is completely to the left of the other.
The ordered graph $G(h)$ is the incomparability graph of $P(h)$, meaning that there is an edge between vertices $i$ and $j$ iff neither $i <_h j$ nor $i >_h j$.
Thus, the class of ordered graphs of the form $G(h)$ could also be called the class of unit interval overlap graphs.
We have the following alternate characterization of these ordered graphs.
\begin{proposition}
  If $G = ([n], E, {<})$ is an ordered graph on the set $[n]$ with the usual numeric order, then the following are equivalent:
  \begin{enumerate}[nosep,label=(\arabic*)]
    \item $G = G(h)$ for some Hessenberg function $h$; and
    \item if $\set{i, j}$ is an edge and $i \leq i' < j' \leq j$, then $\set{i', j'}$ is an edge.
  \end{enumerate}
\end{proposition}
\begin{proof}
  {$(1) \Rightarrow (2)$.}
  If $G = G(h)$ and $\set{i, j} \in E$ and $i \leq i' < j' \leq j$, then
  \begin{equation}
    i' < j' \leq j \leq h(i) \leq h(i')
  \end{equation}
  so that $i' < j' \leq h(i')$ and $\set{i', j'}$ is an edge of $G(h)$ as well.

  \medskip\noindent
  {$(2) \Rightarrow (1)$.}
  Define the function $h \colon [n] \to [n]$ by setting $h(i)$ to be the largest $j$ such that $i < j$ and $\set{i, j}$ is an edge, or $h(i) = i$ if there is no such edge.
  Then, the function $h$ is extensive by construction. The function $h$ is also increasing, making it a Hessenberg function:
  if $h(i) = i$ or $h(i) = i+1$, then $h(i) \leq i+1 \leq h(i+1)$;
  otherwise $\set{i, h(i)}$ is an edge and $i \leq i+1 < h(i) \leq h(i)$, so that $\set{i+1, h(i)}$ is also an edge and $h(i+1) \geq h(i)$ by maximality.
  Again by maximality, we have
  \begin{equation}
    \set{i, j} \in E
    \quad \Longrightarrow \quad
    i < j \leq h(i),
  \end{equation}
  and by condition $(2)$ for the edge $\set{i, h(i)}$ if $h(i) \neq i$, we have
  \begin{equation}
    \set{i, j} \in E
    \quad \Longleftarrow \quad
    i < j \leq h(i),
  \end{equation}
  so that in fact $G = G(h)$ for the Hessenberg function $h$ as required.
\end{proof}

Since condition $(2)$ is a kind of closure condition, it follows that the class of ordered graphs of the form $G(h)$ is closed under taking lexicographic unions and restrictions, and it spans a Hopf subalgebra of $\og$.
\begin{corollary}
  Let $\dyck$ be the subspace of $\og$ spanned by the ordered graphs of the form $G(h)$ for all Hessenberg functions $h \colon [n] \to [n]$ for all $n \geq 0$.
  Then, $\dyck$ is closed under the multiplication and comultiplication maps of $\og$, so that it is a graded-connected Hopf subalgebra.
\end{corollary}

We call $\dyck$ the Hopf algebra of Dyck paths, since Hessenberg functions have yet another representation, as Dyck paths.
These are paths which start on the horizon, take unit steps either Northeast or Southeast, always stay weakly above the horizon, and end on the horizon.
In this representation, multiplication is simply concatenation.

\section{A bit of geometry: Hessenberg varieties}\label{sec:hess}
Having defined Hessenberg functions in \autoref{sec:dyck}, we can now define Hessenberg varieties, which live inside the full flag variety of $\CC^n$.
We give the strict minimum needed to situate the claims in this paper with respect to the literature on Hessenberg varieties, and to formulate the Shareshian--Wachs conjecture, since after this section we will be working with Tymoczko's very explicit algebro-combinatorial description of the equivariant cohomology rings, rather than with the geometric objects themselves.

Fix $n \geq 0$.
The \emph{flag variety} $\flag(\CC^n)$ consists of the complete flags in $\CC^n$, that is, sequences of nested subspaces of the form $F_\bullet = (F_1 \subset F_2 \subset \cdots \subset F_n)$ with $\dim(F_i) = i$ for $i = 1, 2, \ldots, n$. Given an $n \times n$ complex matrix $M$ and a Hessenberg function $h \colon [n] \to [n]$, the corresponding \emph{Hessenberg variety} is the subvariety of $\flag(\CC^n)$ defined by
\begin{equation}
  \hess(M, h) = \set{F_\bullet \mid \text{$MF_i \subseteq F_{h(i)}$ for each $i = 1, 2, \ldots, n$}}.
\end{equation}
Note that there are many choices of $M$ which give the same set of flags; for example, $\hess(M, h)$ and $\hess(M', h)$ are the same if $M' = aM + bI$, where $a$ is a nonzero complex number, $b$ is any complex number, and $I$ is the identity matrix. Also, many choices of $M$ give isomorphic subvarieties; in particular, if $M$ and $M'$ are conjugate matrices, then $\hess(M, h)$ and $\hess(M', h)$ are related by a change of basis of the ambient space $\CC^n$. Thus, it makes sense to consider choices of $M$ with a specified Jordan block structure, or with conditions on the eigenvalues. Various adjectives get attached to Hessenberg varieties based on such restrictions, so that $\hess(M, h)$ is called:
\begin{itemize}[nosep,label=--]
  \item \emph{regular} if every Jordan block of $M$ has a different eigenvalue;
  \item \emph{nilpotent} if every Jordan block of $M$ has eigenvalue 0;
  \item \emph{semisimple} if $M$ has $n$ Jordan blocks of size 1.
\end{itemize}
These adjectives can of course be combined:
\begin{itemize}[nosep,label=--]
  \item $\hess(M, h)$ is \emph{regular nilpotent} if $M$ has a single Jordan block of size $n$, with eigenvalue 0. This is the case studied in \cite{abe-harada-horiguchi-masuda15}.
  \item $\hess(M, h)$ is \emph{minimal nilpotent} if $M$ has one Jordan block of size 2 and $n-2$ Jordan blocks of size 1, all with eigenvalue 0. This is the case studied in \cite{abe-crooks15}.
  \item $\hess(M, h)$ is \emph{regular semisimple} if $M$ is diagonalizable, with $n$ distinct eigenvalues. This is the case that the Shareshian--Wachs conjecture is concerned with, and seems to be the most studied.
\end{itemize}

For the rest of the section, let $M$ be a fixed diagonal matrix with distinct diagonal entries, which we will write as $D$ to emphasize that we are focusing on the case where $\hess(D, h)$ is regular semisimple.
Let $T$ be the group of invertible diagonal matrices, which is isomorphic to the complex torus $(\CC^*)^n$.
Every matrix in $T$ commutes with $D$, so the variety $\hess(D, h)$ is invariant under the action of the group $T$ on $\flag(\CC^n)$.
Thus, we can consider the equivariant cohomology ring, which we will abbreviate as
\begin{equation}
  H^*_T(h) = H^*_T\paren[\big]{\hess(D, h)},
\end{equation}
and which is a graded-connected algebra over $\CC$.
Using the tools of GKM theory \cite{goresky-kottwitz-macpherson98}, Tymoczko has given an explicit description of this ring \cite{tymoczko07b}.
GKM theory also implies that $H^*_T(h)$ is a free module over the equivariant cohomology ring $H^*_T(\text{pt})$ of a point, which is isomorphic to a polynomial ring in $n$ indeterminates, say $\CC[\LL]$, where $\LL = (L_1, L_2, \ldots, L_n)$ is a set of $n$ indeterminates.
Thus, we have a tower of graded-connected $\CC$-algebras
\begin{equation}\label{eq:ltower}
  \CC \subseteq \CC[\LL] \subseteq H^*_T(h),
\end{equation}
where $H^*_T(h)$ is a free module over the subring $\CC[\LL]$.
A further feature of GKM theory is that the ordinary cohomology ring can be recovered as a ring quotient from the equivariant cohomology:
\begin{equation}
  H^*(h) = H^*_T(h) / \gen{\LL} H^*_T(h).
\end{equation}
This amounts to setting each of the indeterminates $L_1, \ldots, L_n$ to zero.

In \cite{tymoczko07b}*{Section 3.1}, Tymoczko defines the \emph{dot action}, which is an action of the symmetric group $S_n$ on $H^*_T(h)$.
Note that this action is $\CC$-linear and respects the grading on $H^*_T(h)$.
In the tower of \eqref{eq:ltower}, the dot action fixes $\CC$ pointwise, and it sends $\CC[\LL]$ to itself; explicitly, if $w \colon [n] \to [n]$ is a permutation, then under the dot action $w \cdot L_i = L_{w(i)}$.
This means that the action of $w \in S_n$ on $H^*_T(h)$ is $\CC$-linear, and \emph{twisted} $\CC[\LL]$-linear.
Still, the following notion of the \emph{graded trace} of $w$ is well-defined:
we will write
\begin{equation}
  \tr_q\paren[\big]{w, H^*_T(h), \CC[\LL]} =
  \sum_i q^{\deg(e_i)} (\text{coefficient of $e_i$ in $w \cdot e_i$})
  \in \CC[[q]],
\end{equation}
where the elements $e_i$ form a homogeneous basis of $H^*_T(h)$ over $\CC[\LL]$, $\deg(e_i)$ is the degree of $e_i$, and the result is a formal power series in the indeterminate $q$ with coefficients in $\CC$.
The careful reader will have spotted that this `definition' packs quite a few implicit assumptions:
\begin{itemize}[nosep,label=--]
  \item Naively, the coefficient of $e_i$ in $w \cdot e_i$ should be an element of $\CC[\LL]$, not $\CC$. However, the dot action preserves the grading, so the coefficient is in the degree-zero part of $\CC[\LL]$, which is $\CC$.
  \item In order to speak of a formal power series, there should only be finitely many contributions to the coefficient of each power of $q$. This is guaranteed because each homogeneous graded piece of $H^*_T(h)$ is finite-dimensional over $\CC$. In fact, $H^*_T(h)$ is finite-dimensional over $\CC[\LL]$, so the result is a polynomial in $q$, but we will use this notation for cases where the result is genuinely a power series.
  \item We are calling this the trace of $w$ over $\CC[\LL]$, but $w$ is not a $\CC[\LL]$-linear map! However, it can be checked that this definition does not depend on the choice of homogeneous basis elements $e_i$.
\end{itemize}
Given this notion of trace, we can define the \emph{graded Frobenius characteristic} of the action of $\sll$ on $H^*_T(h)$ over $\CC[\LL]$ as
\begin{equation}
  \frob_q\paren[\big]{\sll, H^*_T(h), \CC[\LL]} = \frac{1}{n!} \sum_{w \in S_n} \tr_q\paren[\big]{w, H^*_T(h), \CC[\LL]} p_{(\text{cycle type of $w$})},
\end{equation}
where $p_\lambda \in \sym$ is a power-sum symmetric function.
Let $\omega \colon \sym \to \sym$ be the usual involution, which is the Hopf map defined by $\omega(p_k) = (-1)^{k+1} p_k$ for $k = 1, 2, \ldots$.
With all of these definitions, we can finally state the Shareshian--Wachs conjecture:
\begin{conjecture}[\cite{shareshian-wachs12}*{Conjecture 5.3}]
  For all Hessenberg functions $h$,
  \begin{equation}
    \omega\paren[\big]{\frob_q\paren[\big]{\sll, H^*_T(h), \CC[\LL]}} = \csf_q\paren[\big]{G(h)}.
  \end{equation}
\end{conjecture}
In view of \autoref{thm:psi0}, we can rephrase this as
\begin{equation}
  \omega\paren[\big]{\frob_q\paren[\big]{\sll, H^*_T(h), \CC[\LL]}} = \Psi_0\paren[\big]{G(h)},
\end{equation}
which suggests a two-step approach to the proof:
\begin{enumerate}[nosep]
  \item show that $\omega\paren[\big]{\frob_q\paren[\big]{\sll, H^*_T(h), \CC[\LL]}}$, as a function of the Dyck path $G(h)$, respects the multiplication and comultiplication, so that it is a map of graded Hopf algebras from $\dyck$ to $\sym$; then
  \item compute the character for this map by composing it with $\zq$, and show that the result is $\zeta_0$.
\end{enumerate}
This is what we will do, except that we will go through an intermediate step.
We will identify a polynomial subring of $H^*_T(h)$ different from $\CC[\LL]$, which we call $\CC[\RR]$, where $\RR = (R_1, R_2, \ldots, R_n)$ is another set of indeterminates, giving us a second tower of graded $\CC$-algebras,
\begin{equation}\label{eq:rtower}
  \CC \subseteq \CC[\RR] \subseteq H^*_T(h).
\end{equation}
For this tower, $H^*_T(h)$ is again a free module over the subring $\CC[\RR]$, but the dot action fixes $\CC[\RR]$ pointwise: for every permutation $w \in S_n$, we have $w \cdot R_i = R_i$.
Thus, the action of $w$ on $H^*_T(h)$ is simply $\CC[\RR]$-linear, rather than twisted $\CC[\RR]$-linear.
This will make it easier to carry out the two-step process described above for the tower \eqref{eq:rtower}, and show that
\begin{equation}
  \omega\paren[\big]{\frob_q\paren[\big]{\sll, H^*_T(h), \CC[\RR]}} = \Psi_q\paren[\big]{G(h)}.
\end{equation}
The final step will be to show that replacing $\CC[\RR]$ by $\CC[\LL]$ in the tower has the same effect as replacing the character $\zeta_q$ by the character $\zeta_0$.

Given that the ring $\CC[\LL]$ already has a geometric interpretation, in terms of the tangent of the torus $T$, and that the ring $\CC[\RR]$ seems to fit fairly naturally into the algebraic side of the picture, this raises the following question:
\begin{question}
  What, if anything, does the ring $\CC[\RR]$ correspond to on the Hessenberg variety $\hess(D, h)$, geometrically speaking?
\end{question}

\section{Tymoczko's rings}\label{sec:tymo}
In this section, we give Tymoczko's combinatorial construction of the equivariant cohomology ring $H^*_T(h)$ of the regular semisimple Hessenberg variety $\hess(D, h)$ in terms of the moment graph from GKM theory, which comes equipped with her `dot action' by the symmetric group $S_n$ (see \cite{tymoczko07b} for details).
For a fixed $n$, the rings $H^*_T(h)$ for all Hessenberg functions $h \colon [n] \to [n]$ can all be realized as subrings of a single ring $\TT$, which we describe first.
We then define the subring $\TT_h$, which is isomorphic to $H^*_T(h)$, for each Hessenberg function $h$.
In terms of notation, we write $\LL = (L_1, \ldots, L_n)$ for the indeterminates $(t_1, \ldots, t_n)$ from \cite{tymoczko07b}, to emphasize that these indeterminates naturally act on $\TT$ from the left, and introduce new indeterminates $\RR = (R_1, \ldots, R_n)$ which act naturally from the right; it is our hope that this distinction between left and right indeterminates will help the reader avoid confusion when performing explicit computations, especially once the actions of the symmetric groups $S(\LL)$ and $S(\RR)$ on $\TT$ are thrown in the mix.

\subsection{The big ring}
Fix $n \geq 0$; let $\LL = (L_1, \ldots, L_n)$ and $\RR = (R_1, \ldots, R_n)$ be two sets of indeterminates.
We will write $\sll$ for the group of permutations of the indeterminates $\LL$, $\srr$ for the group of permutations of the indeterminates $\RR$, and $\slr$ for the set of bijections from $\RR$ to $\LL$; the peculiar notation is to emphasize that $\sll$ acts on $\slr$ on the left by function composition, and $\srr$ acts on the right.
As a ring, and as a left $\CC[\LL]$-module, Tymoczko's ring is the cartesian product
\begin{equation}\label{eq:ltymo}
  \TT = \prod_{\beta\in\slr} \CC[\LL].
\end{equation}
We will write $\one_\beta$ for the element of $\TT$ which is 1 at coordinate $\beta$ and 0 at all other coordinates, so that a typical element of $\TT$ is of the form
\begin{equation}\label{eq:ltypical}
  \sum_{\beta\in\slr} f_\beta(L_1, \ldots, L_n) \, \one_\beta,
\end{equation}
and the unit element is $\one = \sum_\beta \one_\beta$.
We say that the element \eqref{eq:ltypical} is \emph{homogeneous} of degree $d$ if each polynomial $f_\beta$ is homogeneous of total degree $d$ in the indeterminates $\LL$.
With this notion of grading, the ring $\TT$ is a graded-connected $\CC$-algebra, since the degree zero part is $\CC\one$, which is one-dimensional.
We will identify $\CC[\LL]$ with the subring $\CC[\LL]\one \subseteq \TT$.
The ring $\TT$ also carries a natural $\CC[\LL]$-linear action by $\srr$, defined by $\one_\beta \cdot w = \one_{\beta\circ w}$ for $w \in \srr$.
The ring $\TT$ can also be identified with the product ring
\begin{equation}\label{eq:rtymo}
  \TT = \prod_{\beta\in\slr} \CC[\RR],
\end{equation}
with a typical element being of the form
\begin{equation}\label{eq:rtypical}
  \sum_{\beta\in\slr} \one_\beta \, g_\beta(R_1, \ldots, R_n),
\end{equation}
since the rings $\CC[\LL]$ and $\CC[\RR]$ are isomorphic.
However, they are isomorphic in many different ways, and the choice of isomorphism here is important.
In fact, we will choose a different isomorphism for each coordinate, since each $\beta \in \slr$ gives a natural way of identifying the indeterminates of $\CC[\LL]$ and $\CC[\RR]$.
In other words, we impose the identity
\begin{equation}
  \one_\beta \, R_i = \beta(R_i) \, \one_\beta
\end{equation}
for every bijection $\beta$ and every indeterminate $R_i$, so that
\begin{equation}
  f_\beta(L_1, \ldots, L_n) = g_\beta\paren[\big]{\beta(R_1), \ldots, \beta(R_n)}
\end{equation}
if the elements \eqref{eq:ltypical} and \eqref{eq:rtypical} are equal.
In the presentation from \eqref{eq:rtymo}, the ring $\TT$ has a natural structure as a right $\CC[\RR]$-module, and a $\CC[\RR]$-linear action by $\sll$ on the left, defined by $w \cdot \one_\beta = \one_{w\circ\beta}$ for $w \in \sll$.
As with $\CC[\LL]$, we will identify $\CC[\RR]$ with the subring $\one\CC[\RR] \subseteq \TT$.
So far, we have the following structures on the ring $\TT$:
\begin{itemize}[nosep,label=--]
  \item it is a graded-connected commutative $\CC$-algebra;
  \item it has a left action by $\CC[\LL]$;
  \item it has a left action by $\sll$ (this is in fact the \emph{dot action});
  \item it has a right action by $\CC[\RR]$; and
  \item it has a right action by $\srr$.
\end{itemize}
How do all of these structures interact?
\begin{itemize}[nosep,label=--]
  \item Everything in sight is $\CC$-linear and respects the grading;
  \item both left actions \emph{commute} with both right actions;
  \item the action by $\sll$ \emph{distributes} over the action by $\CC[\LL]$, so that $w \cdot (L_i \, \one_\beta) = w(L_i) \, \one_{w\circ\beta}$, making it a twisted $\CC[\LL]$-linear action; and
  \item the action by $\srr$ \emph{distributes} over the action by $\CC[\RR]$, so that $(\one_\beta \, R_i) \cdot w = \one_{\beta\circ w} \, w(R_i)$, making it a twisted $\CC[\RR]$-linear action.
\end{itemize}

\subsection{The moment graph}
Consider the set $\slr$ of bijections from $\RR$ to $\LL$.
We will now define a directed graph structure $B$ with $\slr$ as vertex set, and a spanning subgraph of it $M(h)$ for each Hessenberg function $h$.
An \emph{inversion} of $\beta \in \slr$ is a pair $(R_i, R_j)$ with $i < j$ such that the corresponding pair $(L_{i'}, L_{j'}) = (\beta(R_i), \beta(R_j))$ has $i' > j'$.
For every $i < j$, half of the vertices in $\slr$ have $(R_i, R_j)$ as an inversion, and half of the vertices don't.
The vertices from these two sets can be paired up by using the transposition $(R_i \leftrightarrow R_j) \in \srr$;
we will put a directed edge in $B$ from $\beta$ to $\beta \circ (R_i \leftrightarrow R_j)$ for every $\beta$ which doesn't have $(R_i, R_j)$ as an inversion, and say that this edge is labelled by $(R_i, R_j)$.
Doing this for every $i < j$ gives the directed edge set of $B$.

Given a Hessenberg function $h$, recall that the ordered graph $G(h)$ has an edge between $i$ and $j$ when $i < j \leq h(i)$. We will define the directed graph $M(h)$ as the subgraph of $B$ which contains all of its vertices, and only the directed edges which are labelled by $(R_i, R_j)$ for those pairs $i, j$ with $i < j \leq h(i)$. This is the \emph{moment graph} for the regular semisimple Hessenberg variety $\hess(D, h)$.

Note that the underlying undirected graph of $B$ is essentially the Cayley graph on $\slr$ generated by all transpositions in $\srr$, or equivalently, all transpositions in $\sll$. With the given edge orientations, $B$ has a single \emph{source}, the vertex $\beta_0$ such that $\beta_0(R_i) = L_i$ for all $i$, and a single \emph{sink}, the vertex $\beta_1$ such that $\beta_1(R_i) = L_{n+1-i}$ for all $i$, and it is acyclic.

\subsection{Edge conditions}
Given an element
\begin{equation}
  x = \sum_{\beta\in\slr} \one_\beta \, g_\beta(R_1, \ldots, R_n)
\end{equation}
of $\TT$ and a directed edge $\beta \to \beta'$ with label $(R_i, R_j)$, we will say that $x$ \emph{satisfies the edge condition} for this edge if the difference
\begin{equation}
  g_\beta(R_1, \ldots, R_n) - g_{\beta'}(R_1, \ldots, R_n)
\end{equation}
is divisible by $R_i - R_j$.
The reader with a taste for symmetry should note we could have labelled the edge $(L_{i'}, L_{j'}) = (\beta(R_i), \beta(R_j))$ and asked that
\begin{equation}
  f_\beta(L_1, \ldots, L_n) - f_{\beta'}(L_1, \ldots, L_n)
\end{equation}
be divisible by $L_{i'} - L_{j'}$ for the representation of $x$ as
\begin{equation}
  x = \sum_{\beta\in\slr} f_\beta(L_1, \ldots, L_n) \, \one_\beta
\end{equation}
instead; this would actually be an equivalent condition, even though we are dealing with two different identifications (given by $\beta$ and $\beta'$) between the rings $\CC[\LL]$ and $\CC[\RR]$. However, this reader should also note that there is a fundamental $\LL$--$\RR$ asymmetry in the definition of the moment graph $M(h)$.

Note that the elements of $\TT$ which satisfy a given edge condition form a subring of $\TT$ which contains $\CC[\LL]$ and $\CC[\RR]$.

\subsection{Building elements}
Given a moment graph $M(h)$, we will be interested in the subring $\TT_h$ of elements of $\TT$ which satisfy the edge conditions for all edges of $M(h)$.
Thankfully, there is a simple procedure for constructing these elements:
\begin{algorithm}\label{algorithm}
\hspace{0em} 
\begin{enumerate}[nosep,label={(Step~\arabic*)},leftmargin=*]
  \item
    Start with a prospective element
    \begin{equation}
      x = \sum_{\beta\in\slr} \one_\beta \, g_\beta(R_1, \ldots, R_n)
    \end{equation}
    where all of the polynomials $g_\beta$ are unassigned.
  \item\label{step:loop}
    Pick any vertex $\beta'$ such that for all directed edges $\beta \to \beta'$ in $M(h)$, the polynomial $g_\beta$ is already assigned. (At the first step, this might be the source vertex $\beta_0$ defined by $\beta_0(R_i) = L_i$ for all $i$, since it has no incoming edges.)
  \item\label{step:choose}
    Choose any polynomial which satisfies the edge condition for every edge $\beta \to \beta'$ (thus ignoring any edges $\beta' \to \beta$), and assign it to $g_{\beta'}$.
  \item
    While there are still unassigned polynomials, return to \ref*{step:loop}.
\end{enumerate}
\end{algorithm}
This algorithm relies on a few assumptions:
\begin{itemize}[nosep,label=--]
  \item In \ref*{step:loop}, there is always a suitable vertex to be picked, and every vertex will be picked eventually. This is easy to see, given that $M(h)$ is a directed acyclic graph.
  \item In \ref*{step:choose}, no matter what previous choices have been made, there always exists at least one polynomial which satisfies the incoming edge conditions. This is far from trivial; see \cite{tymoczko05}*{Section 6} and references therein for a proof.
\end{itemize}
Given these assumptions, however, it should be clear that every element of $\TT_h$ can be constructed using this procedure.

\subsection{Flow-up vectors}
For each vertex $\beta \in \slr$, there are elements of $\TT_h$ which will be particularly useful for what follows, called \emph{flow-up vectors} for $\beta$ in $\TT_h$.
They are, in a sense, elements whose `leading coefficient' is at coordinate $\beta$ and as small as possible.
They are constructed using \autoref{algorithm} by making these choices:
\begin{itemize}[nosep,label=--]
  \item If there is no path $\beta \to \cdots \to \beta'$ in $M(h)$, then assign $g_{\beta'} = 0$.
  \item At the vertex $\beta$, assign
    \begin{equation}
      g_{\beta} = \prod (R_i - R_j),
    \end{equation}
    where the product is over all edge labels $(R_i, R_j)$ of incoming edges.
  \item Otherwise, when there is a path $\beta \to \cdots \to \beta'$ in $M(h)$, pick a polynomial for $g_{\beta'}$ which is homogeneous of the same degree as $g_\beta$.
\end{itemize}
These choices are compatible with \autoref{algorithm}, so flow-up vectors always exist.
As noted in Tymoczko's work, picking a flow-up vector for each $\beta \in \slr$ gives a homogeneous basis for $\TT_h$ as a graded $\CC[\LL]$-module, thus showing that it is a free module of rank $n!$.
We note here that the same is true over $\CC[\RR]$, with the same basis.

\subsection{The small rings}
Thus, for every Hessenberg function $h$, we have the following structures on the subring $\TT_h$ of Tymoczko's ring $\TT$:
\begin{itemize}[nosep,label=--]
  \item it is isomorphic to the equivariant cohomology ring $H^*_T(h)$ of the regular semisimple Hessenberg variety $\hess(D, h)$;
  \item it is a graded-connected commutative $\CC$-subalgebra of $\TT$, containing the subrings $\CC[\LL]$ and $\CC[\RR]$;
  \item it is a free graded module of rank $n!$ over $\CC[\LL]$;
  \item it is a free graded module of rank $n!$ over $\CC[\RR]$;
  \item it is stable under the dot action of $\sll$ on the left;
  \item it is not generally stable under the action of $\srr$ on the right.
\end{itemize}
The stability under $\sll$ but not $\srr$ comes from the fact that the action of $\sll$ on the directed graph $B$ sends the subgraph $M(h)$ to itself (at least, when ignoring the direction of edges), but in general the action of $\srr$ does not.

Given two Hessenberg functions $h, h'$ with $h \leq h'$ pointwise, we have the inclusion of subgraphs $M(h) \subseteq M(h')$, so that the reverse inclusion of subrings $\TT_h \supseteq \TT_{h'}$ holds.
At one extreme, when $h(i) = i$ for all $i$, there are no edge conditions imposed, and $\TT_h = \TT$.
At the other extreme, when $h(i) = n$ for all $i$, all possible edge conditions are imposed, and it can be seen with some work that $\TT_h$ is the subring of $\TT$ generated by $\CC[\LL]$ and $\CC[\RR]$, or equivalently, the ring $\CC[\LL] \otimes_\Lambda \CC[\RR]$, where $\Lambda = \CC[\LL] \cap \CC[\RR] \subseteq \TT$ is the ring of symmetric polynomials in either $\LL$ or $\RR$.

\section{A few facts about symmetric functions}\label{sec:sym}
We will need just a few facts about symmetric functions, especially in relation to representations of symmetric groups, which we record here.
For much more detail on the topic, see \cite{macdonald95}*{Part I}.

The ring of symmetric function $\sym$ is a subring of the $\qsym$.
It consists of those quasisymmetric functions for which the coefficient of $M_\alpha$ and $M_{\alpha'}$ are equal whenever the list of positive integers $\alpha'$ can be obtained from the list $\alpha$ by reordering its entries.
In other words, it consists of all $\CC(q)$-linear combinations of the \emph{monomial} symmetric functions
\begin{equation}
  m_\lambda = \sum_{\alpha\sim\lambda} M_\alpha,
\end{equation}
where $\lambda = (\lambda_1 \geq \lambda_2 \geq \cdots \geq \lambda_r)$ is a \emph{partition}, that is, a weakly decreasing list of positive integers; and $\alpha \sim \lambda$ when the sorted rearrangement of $\alpha = (\alpha_1, \alpha_2, \ldots, \alpha_r)$ is $\lambda$.
There are several other combinatorially significant bases of $\sym$, but for our purposes we will only need the basis of \emph{power-sum} symmetric functions $p_\lambda$, defined by
\begin{equation}
  p_{(k)} = m_{(k)} = M_{(k)}
\end{equation}
when $\lambda = (k)$ is a partition with a single part, and multiplicatively by
\begin{equation}
  p_\lambda = p_{(\lambda_1)} \cdot p_{(\lambda_2)} \cdot \cdots \cdot p_{(\lambda_r)}
\end{equation}
when $\lambda = (\lambda_1 \geq \lambda_2 \geq \cdots \geq \lambda_r)$ has many parts.
The comultiplication on the basis of power-sum symmetric functions acts very nicely, since
\newlength{\mywidth}\settowidth{\mywidth}{$p_{(k)}$}
\newcommand{\mydiag}{\makebox[\mywidth][c]{$p_{(k)}$}}
\newcommand{\myoffd}{\makebox[\mywidth][c]{1}}
\begin{equation}\begin{aligned}
  \Delta_r\paren[\big]{p_{(k)}}
    &= \mydiag \otimes \myoffd \otimes \cdots \otimes \myoffd \\
    &\mathrel{+} \myoffd \otimes \mydiag \otimes \cdots \otimes \myoffd \\
    &\vdotswithin{=} \\
    &\mathrel{+} \myoffd \otimes \myoffd \otimes \cdots \otimes \mydiag,
\end{aligned}\end{equation}
where $p_{(k)}$ appears on the main diagonal of this array and ones appear everywhere else.
In particular, $\sym$ is closed under the multiplication and the comultiplication of $\qsym$, so it is a graded-connected Hopf subalgebra.
Also, if follows that given any infinite sequence $c = (c_1, c_2, c_3, \ldots)$ of coefficients in the ground field $\CC(q)$, the map defined by
\begin{equation}\begin{aligned}
  \varphi_c \colon \sym &\to \sym
  p_{(k)} &\mapsto c_k p_{(k)}
\end{aligned}\end{equation}
can be extended uniquely to be $\CC(q)$-linear, multiplicative, and comultiplicative, so that it is a graded Hopf endomorphism.
As examples of a few maps of this type, we have:
\begin{itemize}[nosep,label=--]
  \item the identity map $\id$, for which $p_{(k)} \mapsto p_{(k)}$;
  \item the antipode $S$, for which $p_{(k)} \mapsto -p_{(k)}$;
  \item the Eulerian map $E_t$ for $t \in \CC(q)$, for which $p_{(k)} \mapsto t^k p_{(k)}$; and
  \item the involution $\omega$, for which $p_{(k)} \mapsto (-1)^{k+1} p_{(k)}$.
\end{itemize}
Since $\sym$ is both commutative and cocommutative, it can be shown that the \emph{convolution} product of graded Hopf endomorphisms $\varphi_1, \varphi_2, \ldots, \varphi_r$ of $\sym$, defined by
\begin{equation}
  \varphi_1 * \varphi_2 * \cdots * \varphi_r = \nabla_r \circ (\varphi_1 \otimes \varphi_2 \otimes \cdots \otimes \varphi_r) \Delta_r,
\end{equation}
is also a graded Hopf endomorphism of $\sym$.
For the special case of morphisms defined as above by infinite sequences, we have $\varphi_c * \varphi_d = \varphi_{c+d}$, where $(c+d)$ is the sequence $(c_1+d_1, c_2+d_2, \ldots)$.

Now, consider a $\CC$-linear representation $V$ of the symmetric group $S_n$, that is, a finite-dimensional vector space $V$ over $\CC$ together with a $\CC$-linear action of $S_n$ on its elements.
The \emph{trace} of a permutation $w \in S_n$ on $V$ over $\CC$ is
\begin{equation}
  \tr(w, V, \CC) = \sum_i (\text{coefficient of $e_i$ in $w \cdot e_i$}) \in \CC,
\end{equation}
where the sum is over the elements of $e_i$ of a basis of $V$.
The \emph{Frobenius characteristic} of the action of $S_n$ on $V$ over $\CC$ is the symmetric function
\begin{equation}
  \frob(S_n, V, \CC) = \frac{1}{n!} \sum_{w\in S_n} \tr(w, V, \CC) p_{(\text{cycle type of $w$})} \in \sym,
\end{equation}
where the cycle type of $w$ is the partition $\lambda = (\lambda_1 \geq \lambda_2 \geq \cdots \geq \lambda_\ell)$ which gives the lengths of each cycle of $w$.
Let $\alpha = (\alpha_1, \alpha_2, \ldots, \alpha_r)$ be an $r$-tuple of natural numbers such that $\alpha_1 + \cdots + \alpha_r = n$.
The \emph{Young subgroup} $Y_\alpha \subseteq S_n$ of type $\alpha$ consists of all permutations in $S_n$ which permute the first $\alpha_1$ elements of $[n]$ among themselves, and the next $\alpha_2$ among themselves, and so on.
As a group, it is naturally isomorphic to a cartesian product,
\begin{equation}
  Y_\alpha = S_{\alpha_1} \times S_{\alpha_2} \times \cdots \times S_{\alpha_r}.
\end{equation}
Let $V_1, V_2, \ldots V_r$ be representations of $S_{\alpha_1}, S_{\alpha_2}, \ldots, S_{\alpha_r}$ respectively. Then, their tensor product
\begin{equation}
  V = V_1 \otimes V_2 \otimes \cdots \otimes V_r
\end{equation}
naturally has the structure of a representation of $Y_\alpha$. Let $V'$ be the induced representation of $V$ from $Y_\alpha$ to $S_n$, defined by
\begin{equation}
  V' = \bigotimes_{i=1,\ldots,k} w_i V,
\end{equation}
where $w_1, w_2, \ldots, w_k$ is a full set of coset representatives for $Y_\alpha$ in $S_n$, and each $w_i V$ is an isomorphic copy of $V$.
Then, we have the following relation between the Frobenius characteristics of these representations:
\begin{equation}
  \frob(S_n, V', \CC) = \prod_{i=1,\ldots,r} \frob(S_{\alpha_i}, V_i, \CC),
\end{equation}
where the multiplication is the multiplication as symmetric function.

We have a dual relationship for comultiplication.
Let $V$ be any representation of $S_n$.
For any $r$-tuple $\alpha$ of natural numbers, we can consider the representation of $Y_\alpha$ on $V$, and this breaks up as a tensor product
\begin{equation}
  V|_{Y_\alpha} = V_{\alpha,1} \otimes V_{\alpha,2} \otimes \cdots \otimes V_{\alpha,r},
\end{equation}
where each $V_{\alpha,i}$ is a representation of $S_{\alpha_i}$.
Then, the relationship between Frobenius characteristics for comultiplication is
\begin{equation}
  \Delta_r\paren[\big]{\frob(S_n, V, \CC)}
    = \sum_\alpha \bigotimes_{i=1,\ldots,r} \frob(S_{\alpha_i}, V_{\alpha,i}, \CC),
\end{equation}
where the sum is over all $r$-tuples $\alpha$.

We have a third relationship between representations of $S_n$ and symmetric functions: the Kronecker product.
Given two representations $U, V$ of $S_n$, there is a natural structure of a representation of $S_n \times S_n$ on the tensor product $U \otimes V$ as above, if we act by $S_n$ independently on each factor.
However, there is an equally natural structure of a representation of $S_n$ on $U \otimes V$, where we act by $S_n$ simultaneously on both factors (also known as the diagonal action).
For this representation, we have the relation
\begin{equation}
  \frob(S_n, U \otimes V, \CC) = \frob(S_n, U, \CC) \star \frob(S_n, V, \CC),
\end{equation}
where the \emph{Kronecker product} $\star$ is the bilinear product $\sym \otimes \sym \to \sym$ defined by
\begin{equation}
  p_\lambda \star p_\mu = \begin{cases}
    z(\lambda) p_\lambda &\text{if $\lambda = \mu$}, \\
    0 &\text{otherwise.}
  \end{cases}
\end{equation}
Here, the $z(\lambda)$ is the standard scaling factor for the power-sum symmetric function, which can also be recovered from the fact that the Frobenius characteristic of the trivial one-dimensional representation of $S_n$ is the identity for $\star$.

We also note that for a representation $V$ of $S_n$, the coefficient of $M_{(n)}$ in the symmetric function
\begin{equation}
  \frob(S_n, V, \CC)
\end{equation}
is the dimension of the subspace of $V$ on which $S_n$ acts trivially, and the coefficient of $M_{(n)}$ in the symmetric function
\begin{equation}
  \omega\paren[\big]{\frob(S_n, V, \CC)}
\end{equation}
is the dimension of the subspace of $V$ on which $S_n$ acts as the sign representation, that is, the set of vectors $v \in V$ such that $w \cdot v = v$ for every even permutation $w \in S_n$ and $w \cdot v = -v$ for every odd permutation.

\section{The Frobenius character of the Dot action as a Hopf map}\label{sec:mor}
In this section, we will show that the two $\CC(q)$-linear maps defined by
\begin{align}
  G(h) &\mapsto \frob_q(\sll, \TT_h, \CC[\LL]) \in \sym, \\
  G(h) &\mapsto \frob_q(\sll, \TT_h, \CC[\RR]) \in \sym,
\end{align}
on the basis of the graded-connected Hopf algebra $\dyck$ respect the multiplication and comultiplication maps and the grading of $\dyck$ and $\sym$, so that they are in fact maps of graded Hopf algebras.
This mainly involves giving decompositions of $\TT_h$ as a $\CC$-linear vector space which are compatible with the actions of $\CC[\LL]$, $\CC[\RR]$ and $\sll$.
We will then identify the multiplicative characters which uniquely determines these maps, as per \autoref{sec:qsym}, to show
\begin{align}
  \omega\paren[\big]{\frob_q(\sll, \TT_h, \CC[\LL])} &= \Psi_0\paren[\big]{G(h)}, \\
  \omega\paren[\big]{\frob_q(\sll, \TT_h, \CC[\RR])} &= \Psi_q\paren[\big]{G(h)},
\end{align}
where $\omega \colon \sym \to \sym$ is the usual involution on symmetric functions.

\subsection{Respecting multiplication}
Let $h \colon [n] \to [n]$ be a Hessenberg function, let $n = n_1 + n_2 + \cdots + n_r$, and suppose that we have an isomorphism of ordered graphs
\begin{equation}
  G(h_1) \olex G(h_2) \olex \cdots \olex G(h_r) = G(h),
\end{equation}
where each $h_i \colon [n_i] \to [n_i]$ is a Hessenberg function, so that the $r$-fold multiplication of the ordered graphs $G(h_i)$ in $\dyck$ is
\begin{equation}
  \nabla_r\paren[\big]{G(h_1) \otimes G(h_2) \otimes \cdots \otimes G(h_r)} = G(h).
\end{equation}
Then, to show that multiplication is preserved, we would like to show that
\begin{equation}\label{eq:frob-prod}
  \prod_{i=1,\ldots,r} \frob_q(\sll[i], \TT_{h_i}, \CC[\LL_i]) = \frob_q(\sll, \TT_h, \CC[\LL]),
\end{equation}
and similarly for $\CC[\RR]$.
Here, we have introduced different set of indeterminates $\LL_i = (L_{i,1}, \ldots, L_{i,n_i})$ for each $i$ on the left-hand side.
These indeterminates should be understood as identified with the set of indeterminates $\LL = (L_1, \ldots, L_n)$ on the right-hand side under the lexicographic identification of ordered sets
\begin{equation}
  [n_1] \olex [n_2] \olex \cdots \olex [n_r] = [n].
\end{equation}
Symmetric remarks hold for the indeterminates $\RR$.
We can prove \eqref{eq:frob-prod} and its twin statement for $\CC[\RR]$ at the level of $\sll$-representations, as follows.
Consider the $\CC$-algebra
\begin{equation}
  \TT_{h_1} \otimes \TT_{h_2} \otimes \cdots \otimes \TT_{h_r},
\end{equation}
which is naturally a module over $\CC[\LL]$ and $\CC[\RR]$, equipped with an action of the Young subgroup
\begin{equation}
  Y_\LL = \sll[1] \times \sll[2] \times \cdots \times \sll[r] \subseteq \sll.
\end{equation}
Then, we need to show that $\TT_h$ is the induced representation from the Young subgroup $Y_\LL$ to $\sll$.
This has already been shown by Teff \cite{teff13}, but we include a sketch of the argument here so that the reader can compare and contrast this with the proof that comultiplication is respected.
\begin{lemma}\label{thm:mult-works}
  With the notation and assumptions of this subsection, $\TT_h$ is isomorphic to the induced representation of $\TT_{h_1} \otimes \TT_{h_2} \otimes \cdots \otimes \TT_{h_r}$ from the Young subgroup $Y = \sll[1] \times \sll[2] \times \cdots \times \sll[r]$ to $\sll$.
\end{lemma}
\begin{proof}
  Note that in this case, the moment graph $M(h)$ is disconnected, and each of the orbits of the vertices under the action of the group
  \begin{equation}
    Y_\RR = \srr[1] \times \srr[2] \times \cdots \times \srr[r] \subseteq \srr
  \end{equation}
  is a union of connected components.
  This is because $M(h)$ cannot contain any directed edge labelled by indeterminates from $\RR_i$ and $\RR_j$ with $i \neq j$.
  Consider one of these orbits, say the orbit $\OO$ containing the source vertex $\beta_0 \in \slr$, defined by $\beta_0(R_i) = L_i$ for $i = 1, \ldots, n$.
  When restricted to $\OO$, the directed graph $M(h)$ is isomorphic to the cartesian product
  \begin{equation}
    M(h_1) \times M(h_2) \times \cdots \times M(h_r).
  \end{equation}
  Since the elements of $\TT_{h_1}, \ldots \TT_{h_r}$ and $\TT_h$ are defined in terms of satisfying edge conditions, it follows that the natural $\CC$-linear map
  \begin{equation}
    \TT_{h_1} \otimes \TT_{h_2} \otimes \cdots \otimes \TT_{h_r} \longrightarrow \TT_h|_{\OO},
  \end{equation}
  produces elements which satisfy all edge conditions, where $\TT_h|_{\OO} \subseteq \TT_h$ is the subspace of elements whose coordinate polynomials outside of $\OO$ are zero.
  Furthermore, this map takes tuples of flow-up vectors to flow-up vectors, and it follows that it is an isomorphism.

  Now, note that $\OO$ is also the orbit of $\beta_0$ under the Young subgroup $Y_\LL$, so the subspace $\TT_h|_{\OO}$ is stable under the action of $Y_\LL$.
  If $w_1, w_2, \ldots, w_k$ is a complete set of coset representatives for $Y_\LL$ in $\sll$, then the orbits under the action of $Y_\RR$ are $w_1 \OO, w_2 \OO, \ldots, w_k \OO$, and we have a decomposition
  \begin{equation}
    \TT_h = \bigoplus_{i=1,\ldots,k} w_i \cdot \TT_h|_{\OO}.
  \end{equation}
  This is exactly the construction of the induced representation.
\end{proof}

\subsection{Respecting comultiplication}
Let $h \colon [n] \to [n]$ be a Hessenberg function and let $r \geq 0$. Then, on the side of the Hopf algebra $\dyck$, the definition of the $r$-fold comultiplication is
\begin{equation}
  \Delta_r\paren[\big]{G(h)} = \sum_{\substack{\kappa \colon [n] \to [r] \\ \text{arbitrary}}} q^{(\text{\# ascents of $\kappa$ on $G(h)$})} \, G(h)|_\kappa.
\end{equation}
On the side of the Hopf algebra $\sym$, for the representation of $\sll$ on $\TT_h$, we have the equation
\begin{equation}\label{eq:froby}
  \Delta_r\paren[\big]{\frob_q(\sll, \TT_h, \CC[\LL])} = \sum_\alpha
    \frob_q(Y_{\LL,\alpha}, \TT_h, \CC[\LL])
\end{equation}
and similarly for $\CC[\RR]$, where the sum is over all $r$-tuples $\alpha = (\alpha_1, \ldots, \alpha_r)$ of natural numbers such that $\alpha_1 + \cdots + \alpha_r = n$, and $Y_{\LL, \alpha}$ is the Young subgroup consisting of permutations in $\sll$ which swap the first $\alpha_1$ indeterminates among themselves, the next $\alpha_2$ indeterminates among themselves, and so on.
We can get these two equations to line up better with each other by grouping the colourings $\kappa$ according to their type; we will say that $\kappa$ has type $\alpha$ if the number of times colour $i$ is used is $\alpha_i$ for all $i = 1, \ldots, r$, and then we can write
\begin{equation}\label{eq:ascenty}
  \Delta_r\paren[\big]{G(h)} = \sum_\alpha \sum_{\substack{\kappa \colon [n] \to [r] \\ \text{of type $\alpha$}}} q^{(\text{\# ascents of $\kappa$ on $G(h)$})} \, G(h)|_\kappa.
\end{equation}
To show that comultiplication is preserved by the graded Frobenius characteristic, we will show that the terms with the same $\alpha$ in the sums in \eqref{eq:froby} and \eqref{eq:ascenty} match up.
On the level of representations, we will give a sequence of $\CC$-linear projections
\begin{equation}
  \TT_h \longrightarrow \cdots \longrightarrow 0
\end{equation}
which preserve the actions of $\CC[\LL]$, $\CC[\RR]$ and $Y_{\LL,\alpha}$, where there is a projection step for each colouring $\kappa \colon [n] \to [r]$ of type $\alpha$, and the kernel of this projection is isomorphic to the space
\begin{equation}
  \TT_{h_1} \otimes \TT_{h_2} \otimes \cdots \otimes \TT_{h_r}
\end{equation}
with its degree shifted up by the number of ascents of $\kappa$ on $G(h)$, where
\begin{equation}
  G(h)|_\kappa = \paren[\big]{G(h_1),\, G(h_2),\, \ldots,\, G(h_r)}.
\end{equation}
\begin{lemma}
  With the notation and assumptions of this subsection, there exists a suitable sequence of projections.
\end{lemma}
\begin{proof}
  In contrast with the proof of \autoref{thm:mult-works}, we will consider orbits of the group $Y_{\LL,\alpha}$ on the vertices of the moment graph $M(h)$, rather than orbits of $Y_\RR$.
  We will associate to each of these orbits a specific colouring $\kappa$ of type $\alpha$.
  The type $\alpha = (\alpha_1, \ldots, \alpha_r)$ specifies a grouping of the indeterminates $\LL = (L_1, \ldots, L_n)$ into $r$ groups $\LL_i$, where $\LL_1$ consists of the first $\alpha_1$ indeterminates, $\LL_2$ the next $\alpha_2$ indeterminates, and so on.
  Given a bijection $\beta \in \slr$, there is a corresponding grouping of the indeterminates $\RR$ into $\RR_{i,\beta} = \beta^{-1}(\LL_i)$ for $i = 1, \ldots, r$.
  In fact, the orbit $\OO_\beta$ of $\beta$ under the action of $Y_{\LL,\alpha}$ consists of all $\beta' \in \slr$ with the same grouping of indeterminates $\RR$.
  The colouring $\kappa$ associated to this orbit is the one which assigns colour $i \in [r]$ to $j \in [n]$ if $R_j \in \RR_{i,\beta}$.
  By considering the labels $(R_i, R_j)$ of the directed edges of the moment graph $M(h)$, it can be checked that:
  \begin{itemize}[nosep,label=--]
    \item
      if $\set{i, j}$ is an ascent of the colouring $\kappa$ on $G(h)$, then every vertex of $\OO_\beta$ has an incoming edge labelled $(R_i, R_j)$;
    \item
      if $\set{i, j}$ is an descent of the colouring $\kappa$ on $G(h)$, then every vertex of $\OO_\beta$ has an outgoing edge labelled $(R_i, R_j)$; and
    \item
      if $\set{i, j}$ is a monochromatic edge for the colouring $\kappa$ on $G(h)$, then the vertices of $\OO_\beta$ are paired up by edges labelled $(R_i, R_j)$.
  \end{itemize}
  Thus, the induced subgraph of $M(h)$ on the orbit $\OO_\beta$ is a cartesian product of the form
  \begin{equation}
    M(h_1) \times M(h_2) \times \cdots \times M(h_r),
  \end{equation}
  where the restriction of $G(h)$ to $\kappa$ is
  \begin{equation}
    G(h)|_\kappa = \paren[\big]{G(h_1),\, G(h_2),\, \ldots,\, G(h_r)}.
  \end{equation}
  Also, if there is a directed edge in $M(h)$ from $\beta$ to $\beta'$, then every vertex of $\OO_\beta$ has a directed edge to a vertex of $\OO_{\beta'}$ and vice versa.
  Thus, the quotient graph $M(h) / Y_{\LL,\alpha}$ is a directed acyclic graph.
  Let
  \begin{equation}
    \OO_{\beta_1}, \OO_{\beta_2}, \ldots, \OO_{\beta_k}
  \end{equation}
  be a list of the orbits such that all directed edges between the orbits go from a later orbit to an earlier one in the list.
  Then, the list
  \begin{equation}\begin{aligned}
    I_1 &= \TT_h|_{\OO_{\beta_1}} \\
    I_2 &= \TT_h|_{\OO_{\beta_1} \cup \OO_{\beta_2}} \\
    &\vdotswithin{=} \\
    I_k &= \TT_h|_{\OO_{\beta_1} \cup \cdots \cup \OO_{\beta_k}}
  \end{aligned}\end{equation}
  of subspaces of $\TT_h$ is actually a list of nested ideals, where as before $\TT_h|_V$ is the space of elements of $\TT_h$ whose coordinate polynomials are zero outside of the vertex set $V$.
  The corresponding sequence of projections
  \begin{equation}
    \TT_h
    \longrightarrow \TT_h/I_1
    \longrightarrow \TT_h/I_2
    \longrightarrow \cdots
    \longrightarrow \TT_h/I_k = 0
  \end{equation}
  amounts to first modding out by the coordinates in the orbit $\OO_{\beta_1}$, then by the coordinates in $\OO_{\beta_2}$, and so on until all coordinates have been modded out by.
  Clearly, this is compatible with the actions by $\CC[\LL]$, $\CC[\RR]$ and $Y_{\LL,\alpha} \subseteq \sll$.
  It remains to show that the kernel $K_i$ of the $i$th step is as claimed.

  Let $\OO$ be the $i$th orbit, let $\kappa$ be the corresponding colouring, let
  \begin{equation}
    G(h)|_\kappa = \paren[\big]{G(h_1),\, G(h_2),\, \ldots,\, G(h_r)}
  \end{equation}
  be the decomposition of the ordered graph $G(h)$ according to $\kappa$, let $M(h)|_{\OO}$ be the induced subgraph of the moment graph $M(h)$ to the orbit $\OO$, and let
  \begin{equation}
    M(h)|_{\OO} = M(h_1) \times M(h_2) \times \cdots \times M(h_r)
  \end{equation}
  be the associated decomposition as a cartesian product of directed graphs.
  Then, every vertex in $\OO$ is associated to an $r$-tuple of vertices in $M(h_1) \times \cdots \times M(h_r)$.
  An element of
  \begin{equation}
    \TT_{h_1} \otimes \TT_{h_2} \otimes \cdots \otimes \TT_{h_r}
  \end{equation}
  has a polynomial in
  \begin{equation}
    \CC[\RR_1] \otimes \CC[\RR_2] \otimes \cdots \otimes \CC[\RR_r] = \CC[\RR]
  \end{equation}
  for every $r$-tuple of vertices in $M(h_1) \times \cdots \times M(h_r)$, so there is a natural candidate for a $\CC$-linear map
  \begin{equation}
    \TT_{h_1} \otimes \TT_{h_2} \otimes \cdots \otimes \TT_{h_r} \longrightarrow K_i
  \end{equation}
  which simply translates a tuple of vertices in $M(h_1) \times \cdots \times M(h_r)$ into the corresponding vertex of $\OO$.
  The resulting elements do satisfy all the edge conditions for the subgraph $M(h)|_{\OO}$, but this is not enough; to be a proper element of $K_i$, it should also satisfy all incoming edge conditions $\beta \to \beta'$ where $\beta'$ is in the orbit $\OO$ but $\beta$ is not.
  This can be fixed by multiplying every coordinate by
  \begin{equation}
    \prod_{\text{ascents $\set{i, j}$}} (R_i - R_j)
  \end{equation}
  where the product is over all ascents $\set{i, j}$ of the colouring $\kappa$ on $G(h)$; as noted above, every vertex of the orbit $\OO$ has one incoming edge from outside of the orbit for each ascent.
  After this modification, one can check that a tuple of flow-up vectors gets mapped to a flow-up vector, so the map is an isomorphism.
  The modification also introduces a degree shift by the number of ascents of $\kappa$ on $G(h)$, as required.
\end{proof}

\subsection{Computing one of the characters}
Given the results of this section so far, we can conclude that the $\CC(q)$-linear maps defined by
\begin{align}
  G(h) &\mapsto \frob_q(\sll, \TT_h, \CC[\LL]) \in \sym, \\
  G(h) &\mapsto \frob_q(\sll, \TT_h, \CC[\RR]) \in \sym,
\end{align}
respect the multiplication, the comultiplication and the grading of $\dyck$ and $\sym$, so they are maps of graded Hopf algebras.
Each one can be uniquely identified by the values of the corresponding multiplicative characters, which can be obtained by post-composing the maps by the canonical character $\zq \colon \qsym \to \CC(q)$.
Since the $\sll$-representation $\TT_h$ is a twisted representation over $\CC[\LL]$ but a usual representation over $\CC[\RR]$, it will be easier to compute the multiplicative character in the latter case.
Indeed, in terms of the $\CC[\RR]$-linear representation of $\sll$ on $\TT_h$, the value
\begin{equation}
  \zq\paren[\big]{\frob_q(\sll, \TT_h, \CC[\RR])}
\end{equation}
depends only on the subspace of $\TT_h$ on which $\sll$ acts trivially, and the value
\begin{equation}
  \zq\paren[\big]{\omega\paren[\big]{\frob_q(\sll, \TT_h, \CC[\RR])}}
\end{equation}
depends only on the subspace of $\TT_h$ on which $\sll$ acts according to the sign representation.
\begin{lemma}
  The $\CC[\RR]$-linear subspace of $\TT_h$ on which $\sll$ acts according to the sign representation consists of all $\CC[\RR]$ multiples of the element
  \begin{equation}
    \sum_{\beta\in\slr} (-1)^{(\text{\# inversions of $\beta$})} \one_\beta \, \prod_{\text{edges $\set{i, j}$}} (R_i - R_j),
  \end{equation}
  where the product is over all edges $\set{i, j}$ of the ordered graph $G(h)$.
  Thus,
  \begin{equation}
    \zq\paren[\big]{\omega\paren[\big]{\frob_q(\sll, \TT_h, \CC[\RR])}}
      = q^{(\text{\# edges of $G(h)$})}
      = \zeta_q\paren[\big]{G(h)},
  \end{equation}
  so that for every Hessenberg function $h$, we have
  \begin{equation}
    \omega\paren[\big]{\frob_q(\sll, \TT_h, \CC[\RR])} = \Psi_q\paren[\big]{G(h)}.
  \end{equation}
\end{lemma}
\begin{proof}
  Let $\TT_h^\pm$ be the subspace of $\TT_h$ on which $\sll$ acts according to the sign representation, let
  \begin{equation}
    x = \sum_{\beta\in\slr} \one_\beta \, g_\beta(R_1, \ldots, R_n) \in \TT_h^\pm
  \end{equation}
  be any element of $\TT_h^\pm$, and let
  \begin{equation}
    \varepsilon = \sum_{w\in\sll} (-1)^{(\text{sign of $w$})} w \in \CC[\sll]
  \end{equation}
  be the group algebra element which acts as orthogonal projection onto $\TT_h^\pm$.
  Then, for every transposition $(L_i \leftrightarrow L_j)$ we have
  \begin{equation}
    (L_i \leftrightarrow L_j) \cdot x
      = (L_i \leftrightarrow L_j) \varepsilon \cdot x
      = (-\varepsilon) \cdot x
      = -x,
  \end{equation}
  so that the polynomial coordinates of $x$ satisfy
  \begin{equation}
    g_{(L_i \leftrightarrow L_j) \circ \beta}(R_1, \ldots, R_n) = -g_\beta(R_1, \ldots, R_n).
  \end{equation}
  In other words, there is a single polynomial $g(R_1, \ldots, R_n)$ such that $x$ is of the form
  \begin{equation}
    x = \sum_{\beta\in\slr} (-1)^{(\text{\# inversions of $\beta$})} \one_\beta \, g(R_1, \ldots, R_n).
  \end{equation}
  Now, consider the edge conditions for $x \in \TT_h$.
  Each vertex of the moment graph $M(h)$ is incident to an edge labelled $(R_i - R_j)$ for each edge $\set{i, j}$ of the ordered graph $G(h)$.
  Thus, the element $x$ satisfies all edge conditions exactly when $g(R_1, \ldots, R_n)$ is divisible by
  \begin{equation}
    \prod_{\text{edges $\set{i, j}$}} (R_i - R_j)
  \qedhere\end{equation}
\end{proof}

\subsection{Changing the base ring}
Now let us show that the symmetric functions $\frob_q(\sll, \TT_h, \CC[\LL])$ and $\frob_q(\sll, \TT_h, \CC[\RR])$ are related by a reasonably simple automorphism of $\sym$.

\begin{lemma}
  We have the equations
  \begin{align}
    \frob_q(\sll, \TT_h, \CC)
      &= \frob_q(\sll, \TT_h, \CC[\LL]) \star \frob_q(\sll, \CC[\LL], \CC) \label{eq:lfac} \\
      &= \frob_q(\sll, \TT_h, \CC[\RR]) \star \frob_q(\sll, \CC[\RR], \CC), \label{eq:rfac}
  \end{align}
  where $\star$ is the Kronecker product of symmetric functions.
\end{lemma}
\begin{proof}
  Let $w \in \sll$ be a permutation.
  Let $x_i$ for $i = 1, \ldots n!$ be a homogeneous $\CC[\LL]$-linear basis for $\TT_h$ (for example, a flow-up basis) and let $y_j$ for $j = 1, 2, \ldots$ be a $\CC$-linear basis of $\CC[\LL]$ (for example, the basis of all monomials).
  Then, the elements $x_i y_j$ form a homogeneous $\CC$-linear basis for $\TT_h$.
  For each $i$, let $a_i$ be the $\CC[\LL]$-coefficient of $x_i$ in $w \cdot x_i$; by degree considerations, $a_i$ has degree zero, so in fact it lies in $\CC$.
  For each $j$, let $b_j$ be the $\CC$-coefficient of $y_j$ in $w \cdot y_j$.
  Since $w \cdot (x_i y_j) = (w \cdot x_i) (w \cdot y_j)$, it follows that the coefficient of $x_i y_j$ in $w \cdot (x_i y_j)$ is the product $a_i b_j$.
  Also, we have $q^{\deg(x_i y_j)} = q^{\deg(x_i)} q^{\deg(y_j)}$.
  This holds for all $w, i, j$, so
  \begin{equation}
    \tr_q(w, \TT_h, \CC[\LL]) = \tr_q(w, \TT_h, \CC[\LL]) \tr_q(w, \CC[\LL], \CC),
  \end{equation}
  which is enough to show \eqref{eq:lfac}. By the same argument, \eqref{eq:rfac} holds.
\end{proof}

\begin{lemma}\label{thm:change}
  Let $\id, S, E_q, E_{(1-q)} \colon \sym \to \sym$ be the Hopf endomorphisms defined by
  \begin{align}
    \id(p_{(k)}) &= p_{(k)} \\
    S(p_{(k)}) &= -p_{(k)} \\
    E_q(p_{(k)}) &= q^k p_{(k)} \\
    E_{(1-q)}(p_{(k)}) &= (1-q)^k p_{(k)},
  \end{align}
  so that we have the Hopf endomorphism
  \begin{equation}
    \paren[\big]{\id * (S \circ E_q)}(p_{(k)}) = (1-q^k) p_{(k)},
  \end{equation}
  where $*$ is the convolution product. Then, we have
  \begin{equation}\label{eq:base-change}
    E_{(1-q)}\paren[\big]{\frob_q(\sll, \TT_h, \CC[\LL])} = \paren[\big]{\id * (S \circ E_q)}\paren[\big]{\frob_q(\sll, \TT_h, \CC[\RR])}.
  \end{equation}
\end{lemma}
\begin{proof}
  The action of $\sll$ on $\CC[\LL]$ simply permutes the indeterminates $\LL = (L_1, \ldots, L_n)$, so it is easy to compute the graded trace of an element $w \in \sll$ acting on all monomials in the indeterminates $\LL$;
  a monomial is fixed by $w$ iff for every cycle of $w$, the indeterminates in the cycle have the same exponent in the monomial. Thus, the contribution of $w$ to $\frob_q(\sll, \CC[\LL], \CC)$ is
  \begin{equation}
    \frac{1}{n!} \tr_q(w, \CC[\LL], \CC) = \frac{1}{n!} \cdot \frac{p_{(\lambda_1)}}{1 - q^{\lambda_1}} \cdot \frac{p_{(\lambda_2)}}{1 - q^{\lambda_2}} \cdot \cdots \cdot \frac{p_{(\lambda_\ell)}}{1 - q^{\lambda_\ell}},
  \end{equation}
  where $\lambda = (\lambda_1 \geq \lambda_2 \geq \cdots \geq \lambda_\ell)$ is the cycle type of $w$.
  This means that the effect of the Kronecker multiplication by $\frob_q(\sll, \CC[\LL], \CC)$ in \eqref{eq:lfac} is the inverse of the effect of the morphism $(\id * (S \circ E_q))$ from the statement.

  The action of $\sll$ on $\CC[\RR]$ is even simpler; it is the trivial action. Thus, the contribution of $w \in \sll$ to $\frob_q(\sll, \CC[\RR], \CC)$ is
  \begin{equation}
    \frac{1}{n!} \tr_q(w, \CC[\RR], \CC) = \frac{1}{n!} \cdot \frac{p_{(\lambda_1)}}{(1 - q)^{\lambda_1}} \cdot \frac{p_{(\lambda_2)}}{(1 - q)^{\lambda_2}} \cdot \cdots \cdot \frac{p_{(\lambda_\ell)}}{(1 - q)^{\lambda_\ell}},
  \end{equation}
  where again $\lambda$ is the cycle type of $w$.
  The effect of Kronecker multiplication by $\frob_q(\sll, \CC[\RR], \CC)$ in \eqref{eq:rfac} is then the inverse of the effect of the morphism $E_{(1-q)}$ from the statement of the lemma.
  Equation \eqref{eq:base-change} follows.
\end{proof}

\section{A technical lemma proved by sign-reversing involution}\label{sec:inv}
At this point, we almost have our proof of the Shareshian--Wachs conjecture.
We have identified two maps $\Psi_0, \Psi_q \colon \dyck \to \sym$ of graded Hopf algebras, and proved that for every Hessenberg functions $h$,
\begin{align}
  \Psi_0\paren[\big]{G(h)} &= \csf_q\paren[\big]{G(h)} \\
  \Psi_q\paren[\big]{G(h)} &= \omega\paren[\big]{\frob_q(\sll, \TT_h, \CC[\RR])}
\end{align}
in \autoref{sec:qcsf2} and \autoref{sec:mor}, respectively.
We also have the equation
\begin{multline}
E_{(1-q)}\paren[\big]{\omega\paren[\big]{\frob_q(\sll, \TT_h, \CC[\LL])}} \\ = \paren[\big]{\id * (S \circ E_q)}\paren[\big]{\omega\paren[\big]{\frob_q(\sll, \TT_h, \CC[\RR])}}
\end{multline}
from \autoref{thm:change}, after applying the involution $\omega$.
All that remains is to show that $\Psi_0$ and $\Psi_q$ satisfy the same relation,
\begin{equation}
E_{(1-q)}\paren[\big]{\Psi_0\paren[\big]{G(h)}} = \paren[\big]{\id * (S \circ E_q)}\paren[\big]{\Psi_q\paren[\big]{G(h)}},
\end{equation}
by checking that they have the same multiplicative character, that is,
\begin{equation}
\zq\paren[\big]{E_{(1-q)}\paren[\big]{\Psi_0\paren[\big]{G(h)}}} = \zq\paren[\big]{\paren[\big]{\id * (S \circ E_q)}\paren[\big]{\Psi_q\paren[\big]{G(h)}}}.
\end{equation}
Also, note that it is enough to verify this for all ordered graphs of the form $G(h)$ which are nonempty and connected, since these are the multiplicatively irreducible elements of $\dyck$.
Once this is done, it will follow that
\begin{equation}
  \Psi_0\paren[\big]{G(h)} = \omega\paren[\big]{\frob_q(\sll, \TT_h, \CC[\LL])},
\end{equation}
as required.
Note that the proof below relies on knowing that $\Psi_q\paren[\big]{G(h)}$ lies in $\sym \subset \qsym$ for all Hessenberg functions $h$, which we have only established by going through Hessenberg varieties.
It would be nice to also have more direct combinatorial proof of this fact, like we have for $\Psi_0\paren[\big]{G(h)}$.

\subsection{A direct computation}
We can compute $\zq\paren[\big]{E_{(1-q)}\paren[\big]{\Psi_0\paren[\big]{G(h)}}}$ directly.
The action of the map $E_{(1-q)} \colon \sym \to \sym$ is simply to multiply the $k$th homogeneous graded piece of $\sym$ by $(1-q)^k$ for each $k$.
Let $E'_{(1-q)} \colon \dyck \to \dyck$ be map which multiplies the $k$th homogeneous graded piece of $\dyck$ by $(1-q)^k$ for each $k$.
Since $\Psi_0$ respects the grading, we have
\begin{equation}
\zq\paren[\big]{E_{(1-q)}\paren[\big]{\Psi_0\paren[\big]{G(h)}}}
  = \zq\paren[\big]{\Psi_0\paren[\big]{E'_{(1-q)}\paren[\big]{G(h)}}}
  = \zeta_0\paren[\big]{E'_{(1-q)}\paren[\big]{G(h)}}.
\end{equation}
However, $\zeta_0$ is almost always zero.
In fact, the only nonempty connected ordered graphs of the form $G(h)$ which is nonzero under $\zeta_0$ is the ordered graph $G_1$ with a single vertex. This ordered graph has degree 1, so we have
\begin{equation}
  \zq\paren[\big]{E_{(1-q)}\paren[\big]{\Psi_0(G_1)}} = 1-q.
\end{equation}

\subsection{A combinatorial interpretation}
The situation is more complicated with the map $\paren[\big]{\id * (S \circ E_q)} \colon \sym \to \sym$.
Let $\pi_0$ and $\pi_+$ be the projection onto the homogeneous part of degree zero and the parts of positive degree of $\sym$, respectively.
Then, Takeuchi's formula for the antipode tells us that
\begin{equation}
  S = \sum_{r \geq 0} (-1)^r \underbrace{\pi_+ * \cdots * \pi_+}_{\text{$r$ copies}}.
\end{equation}
Thus, we have
\begin{equation}
  \id * (S \circ E_q) = \sum_{r \geq 0} (-1)^r \id * \underbrace{(\pi_+ \circ E_q) * \cdots * (\pi_+ \circ E_q)}_{\text{$r$ copies}}.
\end{equation}
Also, the canonical character $\zq$ is multiplicative, so
\begin{multline}\label{eq:big-char}
  \zq \circ \paren[\big]{\id * (S \circ E_q)} \\ = \sum_{r \geq 0} (-1)^r \paren[\big]{\zq \otimes \underbrace{(\zq \circ \pi_+ \circ E_q) \otimes \cdots \otimes (\zq \circ \pi_+ \circ E_q)}_{\text{$r$ copies}}} \circ \Delta_{r+1}.
\end{multline}
Recall from \autoref{sec:qcsf2} that the coefficient of the quasisymmetric function $M_\alpha$ in $\Psi_q\paren[\big]{G(h)}$ is
\begin{equation}\label{eq:psiq-coeff}
  \sum_{\substack{\kappa \colon V \to [r] \\ \text{of type $\alpha$}}}
    q^{(\text{\# weak ascents of $\kappa$ on $G$})}.
\end{equation}
Let us compute the value of \eqref{eq:big-char} on $M_\alpha$.
By definition,
\begin{equation}
  \Delta_{r+1}(M_\alpha) = \sum_{\alpha^0, \alpha^1, \ldots, \alpha^r}
    M_{\alpha^0} \otimes M_{\alpha^1} \otimes \cdots \otimes M_{\alpha^r},
\end{equation}
where the sum is over all $(r+1)$-tuples $(\alpha^0, \alpha^1, \ldots, \alpha^r)$ of sequences of natural numbers whose concatenation is the original sequence $\alpha$.
Note that $\zq(M_{\alpha^i}) = 0$ if $\alpha^i$ has more than one part, and $\pi_+(M_{\alpha^i}) = 0$ if $\alpha^i$ is empty.
Thus, the only terms which survive when applying \eqref{eq:big-char} to $M_\alpha$ are the ones where either:
\begin{itemize}[nosep,label=--]
  \item $\alpha$ has $r+1$ parts, and $\alpha^i$ consists of the $(i+1)$st part of $\alpha$; or
  \item $\alpha$ has $r$ parts, $\alpha^0$ is empty, and $\alpha^i$ consists of the $i$th part of $\alpha$.
\end{itemize}
All in all, applying \eqref{eq:big-char} to $M_\alpha$ when $\alpha = (\alpha_1, \ldots, \alpha_r)$ gives two terms:
\begin{equation}
  (-1)^{r-1} q^{\alpha_2 + \cdots + \alpha_r} + (-1)^r q^{\alpha_1 + \alpha_2 + \cdots + \alpha_r}.
\end{equation}
Combining this with \eqref{eq:psiq-coeff}, we get the combinatorial interpretation
\begin{equation}\label{eq:stat}
  \zq\paren[\big]{\paren[\big]{\id * (S \circ E_q)}\paren[\big]{\Psi_q\paren[\big]{G(h)}}}
    = \sum_{r \geq 0} \sum_\kappa (-1)^r q^{\stat(\kappa)},
\end{equation}
where the inner sum is over all colourings $\kappa \colon [n] \to \set{0} \cup [r]$ such that every colour in $[r]$ is used at least once (but the special colour 0 is optional), and the statistic in the exponent of $q$ is
\begin{equation}
  \stat(\kappa)
    = (\text{\# vertices with colour $>0$})
    + (\text{\# weak ascents of $\kappa$ on $G(h)$}).
\end{equation}

\subsection{The sign-reversing involution}
We want to show that most of the time, the sum on the right-hand side of \eqref{eq:stat} is zero.
More precisely, it should be zero whenever $G(h)$ is a connected ordered graph with more than one vertex, and it should be $1-q$ when $G(h) = G_1$ is the ordered graph with a single vertex.
(We leave the very finite case of $G_1$ as an exercise to the reader.)
We will show this by exhibiting a sign-reversing involution on the terms $(-1)^r q^{\stat(\kappa)}$ which preserves the statistic in the exponent of $q$, so that all terms cancel out.
Sadly, however, this involution is defined by a lengthy case analysis.
On the plus side, it will not depend on the ordered graph $G(h)$, other than the assumptions that
\begin{itemize}[nosep,label=--]
  \item it has at least two vertices, and
  \item there is an edge joining the vertices last two vertices, $n-1$ and $n$,
\end{itemize}
which are satisfied for all connected ordered graphs with more than one vertex.
Let $\kappa \colon [n] \to \set{0} \cup [r]$ be a colouring such that every colour in $[r]$ is used at least once, but the special colour 0 may be used or not. We will define a colouring $\kappa' \colon [n] \to \set{0} \cup [r\pm1]$ which uses either one more or one less non-special colour, so that the sign is reversed, and for which $\stat(\kappa) = \stat(\kappa')$. The involution is defined in all cases so that the only edge which could change status between being a weak ascent or a strict descent is the edge joining $n-1$ and $n$, and this is compensated by whether the vertices $n-1$ and $n$ have a positive colour.

\subsubsection{If vertex $n$ has colour 0}
Let vertex $n-1$ have colour $i$.
If this is the only vertex with colour $i$, then to obtain $\kappa'$ recolour $n-1$ with colour $i-1$ and delete colour $i$ from the list of available colours.
If $n-1$ is not the only vertex with colour $i$, then to obtain $\kappa'$ recolour $n-1$ with a new colour $i+\tfrac12$.
In both cases, rename the colours in an order-preserving way so that the set of non-special colours is actually $[r\pm1]$.
Note that these two procedures are inverses of each other.

\subsubsection{If vertex $n$ is the only vertex with colour 1}
We play essentially the same game as in the previous case, except that now we skip over the colour 1 when recolouring $n-1$.
That is, let vertex $n-1$ have colour $i$.
If this is the only vertex with colour $i$, then recolour $n-1$ with the previous colour in the sequence
\begin{equation}\label{eq:seqnot1}
  0, 2, 3, 4, 5, \ldots
\end{equation}
and delete colour $i$ from the list of available colours.
If $n-1$ is not the only vertex with colour $i$, then recolour $n-1$ with a new colour, which should be inserted just before the colour that comes after $i$ in the sequence \eqref{eq:seqnot1}.
As before, after doing this, rename the colours in an order-preserving way so that the set of non-special colours used is $[r\pm1]$.
Again, the two procedures in this case are inverses of each other.

\subsubsection{Otherwise}
In all other cases, we recolour the vertex $n$ rather than $n-1$.
If $n$ is the unique vertex with colour $i$, then recolour it $i-i$, and delete colour $i$.
If $n$ is not the unique vertex with colour $i$, then recolour it $i+\tfrac12$, a new colour.
As always, rename the set of non-special colours used to be $[r\pm1]$ and note that these two procedures are inverses.

\hbadness=9999 

\setlength{\hfuzz}{6pt} 
\end{document}